\documentclass[11pt]{imsart}

\RequirePackage[OT1]{fontenc}
\RequirePackage{amsthm,amsmath,amssymb}
\usepackage{hyperref}
\usepackage{mathtools}
\usepackage{hypernat}
\usepackage{marvosym}
\usepackage[hang,small,bf]{caption}
\newtheorem{theorem}{Theorem}[section]

\newtheorem{lemma}[theorem]{Lemma}

\newtheorem{remark}{Remark}[section]
\newtheorem{example}[theorem]{Example}
\DeclareMathOperator*{\argmin}{arg\,min}
\startlocaldefs
\endlocaldefs

\newcommand{\p}{\mathbb{P}}
\newcommand{\eps}{\varepsilon}
\newcommand{\E}{\mathbb{E}}
\newcommand{\pp}{{\mathcal{P}}}

\newcommand{\R}{\mathbb{R}}
\newcommand{\I}{\mathcal{I}}

\newcommand{\U}{\mathcal{U}}

\newcommand{\F}{\mathcal{F}}
\newcommand{\N}{\mathcal{N}}
\newcommand{\K}{\mathcal{K}}
\newcommand{\Z}{\mathbb{Z}}
\newcommand{\W}{\mathbb{W}}
\newcommand{\h}{\mathbb{H}}

\newcommand{\B}{\mathbb{B}}
\newcommand{\M}{\mathcal{M}}

\newcommand{\X}{\mathcal{X}}
\newcommand{\C}{\mathcal{C}}

\newcommand\norm[1]{\lVert#1\rVert}
\def\qt#1{\qquad\text{#1}}

\begin{document}

\begin{frontmatter}
\title{Nonparametric Shape-restricted Regression}
\runtitle{Shape-restricted Regression}
%\thankstext{T1}{Footnote to the title with the `thankstext' command.}

\begin{aug}
\author{\fnms{Adityanand} \snm{Guntuboyina}\thanksref{t1}\ead[label=e1]{aditya@stat.berkeley.edu}}
\and
\author{\fnms{Bodhisattva} \snm{Sen}\thanksref{t2}\ead[label=e2]{bodhi@stat.columbia.edu}}

%\author{\fnms{Third} \snm{Author}\thanksref{t1}
%\ead[label=e3]{third@somewhere.com}
%\ead[label=u1,url]{www.foo.com}}

\thankstext{t1}{Supported by NSF CAREER Grant DMS-16-54589.}
\thankstext{t2}{Supported by NSF Grants DMS-17-12822 and AST-16-14743.}
\runauthor{A.~Guntuboyina and B.~Sen}

\affiliation{University of California at Berkeley and Columbia University}

\address{Department of Statistics \\ University of California at Berkeley \\ 423 Evans Hall \\ Berkeley, CA 94720 \\ \printead{e1}}

\address{Department of Statistics \\ Columbia University \\ 1255 Amsterdam Avenue \\ New York, NY 10027 \\ \printead{e2}}
\end{aug}

\begin{abstract}
We consider the problem of nonparametric regression under shape constraints. The main examples include isotonic regression (with respect to any partial order), unimodal/convex regression, additive shape-restricted regression, and constrained single index model. We review some of the theoretical properties of the least squares estimator (LSE) in these problems, emphasizing on the  adaptive nature of the LSE. In particular, we study the behavior of the risk of the LSE, and its pointwise limiting distribution theory, with special emphasis to isotonic regression. We survey various methods for constructing pointwise confidence intervals around these shape-restricted functions. We also briefly discuss the computation of the LSE and indicate some open research problems and future directions. %Although this is mainly a review article, we also present some new  results.  

\end{abstract}


\end{frontmatter}
\section{Introduction}\label{sec:Intro}
In nonparametric shape-restricted regression the observations $\{(x_i,y_i): i=1, \ldots, n\}$ satisfy
\begin{equation}\label{eq:RegMdl}
	y_i = f(x_i) + \varepsilon_i, \qquad \mbox{for } i = 1,\ldots, n,
\end{equation}
where $x_1, \ldots,  x_n$ are design points in some space (e.g., $\R^d$, $d \ge 1$), $\varepsilon_1,\ldots, \varepsilon_n$ are unobserved mean-zero errors (with finite variances), and the real-valued regression function $f$ is unknown but obeys certain known qualitative restrictions like monotonicity, convexity, etc. Let $\F$ denote the class of all such regression functions. Letting $\theta^* := (f(x_1),\ldots, f(x_n))$, $Y := (y_1,\ldots, y_n)$ and $\eps := (\varepsilon_1,\ldots, \varepsilon_n)$, model~\eqref{eq:RegMdl} may be rewritten as 
\begin{equation}\label{eq:SeqMdl}
Y = \theta^* + \eps,
\end{equation}
and the problem is to estimate $\theta^*$ and/or $f$ from $Y$, subject to the constraints imposed by the properties of $\F$. %In this paper, for simplicity, we assume that $\varepsilon_1,\ldots, \varepsilon_n$ are i.i.d.~$N(0,\sigma^2)$.
The constraints on the function class $\F$ translate to constraints on $\theta^*$ of the form $\theta^* \in \C$, where 
\begin{equation}\label{eq:C}
\C:= \big\{(f(x_1), \dots, f(x_n)) \in \R^n: f \in \F\big\}
\end{equation}
is a subset of $\R^n$ (in fact, in most cases, $\C$ will be a closed convex cone). In the following we give some examples of shape-restricted regression. 

\begin{example}[Isotonic regression]\label{ex:IsoMdl} Probably the most studied shape-restricted regression problem is that of estimating a monotone (nondecreasing) regression function $f$ when $x_1 <\ldots < x_n$ are the univariate design points. In this case, $\F$ is the class of all nondecreasing functions on the interval  $[x_1, x_n]$, and the constraint set $\C$ reduces to
\begin{equation}\label{eq:IsoCone}
\I := \{(\theta_1, \ldots, \theta_n) \in \R^n :\theta_1 \le \ldots \le \theta_n\},
\end{equation} 
which is a closed convex cone in $\R^n$ ($\I$ is defined through $n - 1$ linear constraints). The above problem is typically known as isotonic regression and has a long history in statistics; see e.g.,~\cite{Brunk55, AyerEtAl55, vanEeden56}.
\end{example}

\begin{example}[Order preserving regression on a partially ordered set]\label{ex:RegPartialOrder} Isotonic regression can be easily extended to the setup where the covariates take values in a space $\X$ with a partial order $\precsim$\footnote{A partial order is a binary relation $\precsim$ that is reflexive ($x \precsim x$ for all $x \in \X$), transitive ($u, v, w \in \X, \; u \precsim v$ and $v \precsim w$ imply $u \precsim w$), and antisymmetric ($u , v  \in \X, \; u \precsim v$ and $v \precsim u$ imply $u  = v$).}; see e.g.,~\cite[Chapter 1]{RWD88}. % Let $\X := \{x_1,\ldots, x_n\}$ be the set of $n$ distinct design points with a {\it partial order} $\precsim$ 
A function $f:\X \to \R$ is said to be {\it isotonic} (or order preserving) with respect to the partial order $\precsim$ if for every pair $u,v  \in \X$,  $$ u \precsim v \;\;\Rightarrow \;\; f(u) \le f(v).$$ For example, suppose that the predictors take values in $\R^2$ and the partial order $\precsim$ is defined as $(u_1, u_2) \precsim (v_1,v_2)$ if and only if $ u_1 \le v_1$ and $u_2 \le v_2$. This partial order leads to a natural extension of isotonic regression to two dimensions; see e.g.,~\cite{HPW73, RW75, CGS17}. One can also consider other partial orders; see e.g.,~\cite{Stout14, Stout15} and the references therein for isotonic regression with different partial orders. We will introduce and study yet another partial order  in Section~\ref{sec:Open}. %Another recent paper on isotonic regression under orderings induced by graphs is \cite{CL17}. 

Given data from model~\eqref{eq:RegMdl}, the goal is to estimate the unknown regression function $f:\X \to \R$ under the assumption that $f$ is order preserving (with respect to the partial order $\precsim$). The restrictions imposed by the partial order $\precsim$ constrain $\theta^*$ to lie in a closed convex cone $\C$ which may be expressed as
\begin{equation*}\label{eq:PO-Cons}
\{(\theta_1,\ldots, \theta_n)  \in \R^n: \theta_i \le \theta_j \mbox{ for every } i,j \mbox{ such that } x_i \precsim x_j \}.
\end{equation*} 

\end{example}

\begin{example}[Convex regression]\label{ex:ConvexReg} Suppose that the underlying regression function $f: \R^d \to \R$ ($d \ge 1$) is known to be convex, i.e., for every $u,v \in \R^d$, \begin{equation}\label{eq:CvxDef}
f(\alpha u + (1-\alpha) v) \le \alpha f(u) + (1-\alpha) f(v), \quad \mbox{ for every } \alpha \in (0,1).
\end{equation}
Convexity appears naturally in many applications; see e.g.,~\cite{Hildreth54, K08, DT17} and the references therein. The convexity of $f$ constrains $\theta^*$ to lie in a (polyhedral) convex set $\C \subset \R^n$ which, when $d=1$ and the $x_i$'s are ordered, reduces to
\begin{equation}\label{eq:CvxReg}
\K := \left\{(\theta_1,\ldots, \theta_n) \in \R^n: \frac{\theta_2 - \theta_1}{x_2 - x_1} \le \ldots \le \frac{\theta_n - \theta_{n-1}}{x_n - x_{n-1}} \right\},
\end{equation}
whereas for $d \ge 2$ the characterization of $\C$ is more complex; see e.g.,~\cite{SS11}.  

Observe that when $d=1$, convexity is characterized by nondecreasing derivatives (subgradients). This observation can be used to generalize convexity to $k$-monotonicity ($k \ge 1$): a real-valued function $f$ is said to be $k$-monotone if its $(k-1)$'th derivative is monotone; see e.g.,~\cite{M91, CGS15}.  For equi-spaced design points in $\R$, this restriction constrains $\theta^*$ to lie in the set $$\{\theta \in \R^n: \nabla^k \theta \geq 0\}$$ where $\nabla:\R^n 
  \rightarrow \R^n$ is given by $\nabla(\theta) := (\theta_2 -
  \theta_1, \theta_3 - \theta_2, \dots, \theta_n - \theta_{n-1},
  0)$ and $\nabla^k$ represents the $k$-times composition of $\nabla$. Note that the case $k=1$ and $k=2$ correspond to isotonic and convex regression, respectively. 

 %In fact, $\C$ can be expressed as the projection of the higher-dimensional polyhedron
%\begin{equation}\label{eq:higher_Q}
%\Q:=\{(\bxi,\bt) \in \R^{dn +n} : A \bxi + B \bt \leq \bc\}
%\end{equation}
%onto the space of $\bt$, where $\bxi:=[\bxi_1^\top, \ldots, \bxi_n^\top]^\top $ is the auxiliary vector representing the subgradient of $f(\bx_j)$, for $j=1, \ldots, n$, and $A$, $B$ and $\bc$ are suitable matrices; see Section~\ref{sec:DFProjPoly} for the details. 
\end{example}

\begin{example}[Unimodal regression]\label{ex:UniReg}
In many applications $f$, the underlying regression function, is known to be unimodal; see e.g.,~\cite{Frisen86, Chat16} and the references therein. Let $\I_m$, $1 \le m \le n$, denote the convex set of all unimodal vectors (first decreasing and then increasing) with mode at position $m$, i.e., $$\I_m := \{(\theta_1,\ldots, \theta_n) \in \R^n: \theta_1 \ge \ldots \ge \theta_m \le \theta_{m+1} \le \ldots \le \theta_n\}.$$ Then, the unimodality of $f$ constrains $\theta^*$ to belong to $\U := \cup_{m=1}^n \I_m$. Observe that now $\U$ is not a convex set, but a union of $n$ convex cones.
\end{example}

\begin{example}[Shape-restricted additive model]\label{ex:AddMdl}
In an additive regression model one assumes that $f:\R^d \to \R$ ($d \ge 1$) depends on each of the predictor variables in an additive fashion, i.e., for $(u_1,\ldots, u_d) \in \R^d$, $$f(u_1,\ldots, u_d) = \sum_{i=1}^d f_i(u_i),$$ where $f_i$'s are one-dimensional functions and $f_i$ captures the influence of the $i$'th variable. Observe that the additive model generalizes (multiple) linear regression. If we assume that each of the  $f_i$'s are shape-constrained, then one obtains a shape-restricted additive model; see e.g.,~\cite{bacchetti89, Mammen07, Meyer13, CS16} for a study of some possible applications, identifiability and estimation in such a model. 

%can be done using the backfitting algorithm 
\end{example}
\begin{example}[Shape-restricted single index model]\label{ex:SIM} 
In a single index regression model one assumes that the regression function $f:\R^d \to \R$ takes the form
\begin{equation*}\label{eq:SIM}
f(x) = m(x^\top \beta^*), \quad \mbox{ for all } \, x \in \R^d,
\end{equation*} 
where $m:\R \to \R$  and $\beta^* \in \R^d$ are unknown. Single index models are popular in many application areas, including econometrics and biostatistics (see e.g.,~\cite{Powelletal89, liracine07}), as they circumvent the curse of dimensionality encountered in estimating the fully nonparametric regression function by assuming that the link function depends on $x$ only through a one-dimensional projection, i.e., $x^\top\beta^*$. Moreover, the coefficient vector $\beta^*$ provides interpretability. Observe that single index models extend generalized linear models (where the link function $m$ is assumed known). Moreover, as most known link functions are nondecreasing, the monotone single index model (where $m$ is assumed unknown but nondecreasing) arises naturally in applications; see e.g.,~\cite{VANC, groeneboom2016current, 2016arXiv161006026B}. 
\end{example}

Observe that all the aforementioned problems fall under the general area of nonparametric regression. However, it turns out that in each of the above problems one can use classical techniques like least squares and/or maximum likelihood (without additional explicit regularization/penalization) to readily obtain {\it tuning parameter-free estimators} that have attractive theoretical and computational properties. This makes shape-restricted regression different from usual nonparametric regression, where likelihood based methods are generally infeasible. In this paper we try to showcase some of these attractive features of shape-restricted regression and give an overview of the major theoretical advances in this area.

%the additional information in the form shape constraints can be exploited to construct {\it tuning parameter-free} estimators based on classical techniques like least squares and/or maximum likelihood (without explicit regularization/penalization) that have attractive theoretical and computational properties. 

%({\color{blue} \textbf{AG}: I might rephrase this paragraph slightly as follows. It turns out that in each of the aforementioned problems, one can use classical techniques like least squares and/or maximum likelihood (without additional explicit regularization) to readily obtain tuning parameter-free estimators that have attractive theoretical and computational properties. This makes shape-restricted regression different from usual nonparametric regression where likelihood based methods are generally infeasible.} )

Let us now introduce the estimator of $\theta^*$ (and $f$) that we will study in this paper. 
%where the known shape restriction on $f$ translates to linear constraints on $\theta^*$ whereby $\theta^* \in \C$ for some suitable closed convex set $\C$. \newline
%\subsection{Least squares estimation}
The {\it least squares estimator} (LSE) $\hat \theta$ of $\theta^*$ in shape-restricted regression is defined as the projection of $Y$ onto the set $\C$ (see~\eqref{eq:C}), i.e., \begin{equation}\label{eq:LSE}
\hat \theta := \arg \min_{\theta \in \C} \|Y - \theta\|^2,
\end{equation}
where $\|\cdot\|$ denotes the usual Euclidean norm in $\R^n$. If $\C$ is a closed convex set  then $\hat \theta \in \C$ is unique and is characterized by the following condition:
\begin{equation}\label{eq:Charac}
\langle Y - \hat \theta, \theta - \hat \theta\rangle \le 0, \quad \mbox{for all } \theta \in \C,
\end{equation}
where $\langle \cdot, \cdot \rangle$ denotes the usual inner product in $\R^n$; see~\cite[Proposition 2.2.1]{B03}. It is easy to see now that the LSE $\hat \theta$ is tuning parameter-free, unlike most nonparametric estimators. However it is not generally easy to find a closed-form expression for $\hat \theta$. As for estimating $f$, any $\hat f_n \in \F$ that agrees with $\hat \theta$ at the data points $x_i$'s will be considered as a LSE of $f$.

In this paper we mainly review the main theoretical properties of the LSE $\hat \theta$ with special emphasis on its {\it adaptive} nature. The risk behavior of $\hat \theta$ (in estimating $\theta^*$) is studied in Sections~\ref{sec:IsoRisk} and~\ref{sec:Risk} --- Section~\ref{sec:IsoRisk} mainly deals with the isotonic LSE in detail whereas Section~\ref{sec:Risk} summarizes the main results for other shape-restricted problems. In Section~\ref{sec:Inference} we study the pointwise asymptotic behavior of the LSE $\hat f_n$, in the case of isotonic and convex regression, focusing on methods for constructing (pointwise) confidence intervals around $f$. In the process of this review we highlight the main ideas and techniques used in the proofs of the theoretical results; in fact, we give (nearly) complete proofs in some cases.

The computation of the LSE $\hat \theta$, in the various problems outlined above, is discussed in Section~\ref{sec:Compute}. In Section~\ref{sec:Open} we mention a few open research problems and possible future directions. 
%Section~\ref{sec:Other} gives a brief overview of some of the important topics in shape-restricted regression not covered in this paper. 
Although the paper mostly summarizes known results, we also present some new results --- Theorems~\ref{thm:AdapIsoReg},~\ref{thm:L_p},~\ref{pro}, and~\ref{thm:PO}, and Lemma~\ref{lem:AM-Oracle} are new. Appendix~\ref{sec:Proofs} contains some of the detailed proofs of results in this paper.  

There are indeed many other important applications and examples of shape-restricted regression beyond those highlighted so far. We briefly mention some of these below. Shape constrained functions also arise naturally in interval censoring problems (e.g., in the current status model; see~\cite{GW92, HW97}), in survival analysis (e.g., in estimation of monotone/unimodal hazard rates~\cite{HW95}), and in regression models where the response, conditional on the covariate, comes from a regular parametric family (e.g., monotone response models~\cite{Ban07}). It also arises in the study of many inverse problems; e.g., deconvolution problems (see e.g.,~\cite{GW92, Jv09}) and the classical Wicksell's corpuscle problem (see e.g.,~\cite{GJ95, SW12}). There are many applications that involve testing with shape constraints; see e.g.,~\cite{DS01, SM17, Wei17} and the references therein. 

In this paper we will mostly focus on estimation  of the underlying shape-restricted function using the method of least squares. Although this produces tuning parameter-free estimators, the obtained LSEs are not ``smooth". There is also a line of research that combines shape constraints with smoothness assumptions --- see e.g.,~\cite{Muk88, Mammen91, GJW10} (and the references therein) where kernel-based methods have been combined with shape-restrictions, and see~\cite{MT99, Meyer08, VANC, KPS17} where splines are used in conjunction with the shape constraints.

\subsection{Some applications of shape-restricted regression}
Shape-constrained regression has a long history in statistics:~Hildreth~\cite{Hildreth54} considered least squares estimation (i.e., maximum likelihood estimation under Gaussian errors) of production functions under the natural assumption of nonincreasing returns (which implies that the production function is concave and nondecreasing). Around the same time, Brunk~\cite{Brunk55} considered maximum likelihood estimation of a regression function under monotonicity constraints. Since then isotonic regression (under any partial order) has seen many applications in diverse settings: in biology~\cite{Obozinski08}, in dose-response models~\cite{hu2005analysis}, in psychology~\cite{kruskal1964multidimensional}, in genetics~\cite{Luss12}, etc.  

Similarly, convexity or concavity constraints arise natural in many disciplines. Economic theory dictates that utility functions are increasing and concave~\cite{Matzkin91} whereas production functions are often assumed to be concave~\cite{Varian84}. In finance, theory restricts call option prices to be convex and decreasing functions of the strike price~\cite{A-Y03}; in stochastic control, value functions are often assumed to be convex (see~\cite{Keshavarz11}~\cite[Chapter 2]{Balazs16} and~\cite{Shapiro09}); see~\cite{MB09} for some applications of convex regression in optimization (in particular, in linear programming). 

Unimodal regression also arises in many settings; see~\cite{Frisen86} and the references therein. Shape-restricted additive and single-index regression models offer flexible, yet interpretable, statistical procedures for handling multidimensional covariates, and have been extensively used in econometrics, epidemiology and other fields (see~\cite{CS16, Pya15, KPS17} and the references therein).

%{\color{red}\textbf{This subsection looks good. However, a longer list of applications of convex regression can be found in Magnani and Boyd 2009 and Hannah and Dunson 2011. Also you mentioned a PhD thesis by someone in Alberta who mentioned many applications; can we cite some of them here (or cite the thesis)? We also talk about unimodal regression and shape-constrained additive models; it will be good to mention applications of these as well here. }} 

\section{Risk bounds in Isotonic Regression}\label{sec:IsoRisk}
In this section we attempt to answer the following question: ``How good is $\hat \theta$ as an estimator of   $\theta^*$?". To quantify the accuracy of $\hat \theta$ we first need to fix a loss function. Arguably the most natural loss function here is the squared error loss: $\norm{\hat{\theta} - \theta^*}^2/n$. As the loss function is random, we follow the usual approach and study its expectation: 
\begin{equation}\label{eq:Risk}
R(\hat \theta,\theta^*) := \frac{1}{n} \E_{\theta^*} \big[ \|\hat{\theta} - \theta^*\|^2 \big] = \frac{1}{n} \E_{\theta^*} \sum_{i=1}^n \big(\hat{\theta}_i - \theta^*_i \big)^2
\end{equation} 
which we shall refer to as the {\it risk} of the LSE $\hat{\theta}$. We focus on the risk in this paper. It may be noted that upper bounds derived for  the risk usually hold on the loss $\|\hat{\theta} - \theta^*\|^2/n$ as well, with high probability. When $\eps \sim N_n(0, \sigma^2 I_n)$, this is essentially because $\|\hat \theta - \theta^*\|$ concentrates around its mean; see \cite{vanconcentration} and \cite{Bellec15}  for more details on high  probability results.

%Inequality \eqref{eq:CvxLSE-WC} and \eqref{yno} provide upper bounds on $R(\hat{\theta}, \theta^*) = \E_{\theta^*} \|\hat{\theta} - \theta^*\|^2/n$ under the assumption $\eps \sim N_n(0, \sigma^2 I_n)$. These bounds also hold with high probability on $\|\hat{\theta} - \theta^*\|^2/n$. This is essentially because $\|\hat \theta - \theta^*\|$ concentrates around its mean; see \cite{vanconcentration} and \cite{Bellec15}  for more details on high probability results {\color{red} Maybe bring this before}. 

One can also try to study the risk under more general $\ell_p$-loss functions. For $p \geq 1$, let  
\begin{equation}\label{rlp}
  R^{(p)}(\hat{\theta}, \theta^*) := \frac{1}{n} \E_{\theta^*}
  \big[ \|\hat{\theta} - \theta^*\|_p^p\big] = \frac{1}{n} \E_{\theta^*} \sum_{i=1}^n \big|\hat{\theta}_i - \theta^*_i  \big|^p
\end{equation}
where $\norm{u}_p := \left(\sum_{j=1}^n |u_j|^p \right)^{1/p}$, for $u =(u_1,\ldots, u_n) \in \R^n$. We shall mostly focus on the risk for $p = 2$ in this paper but we shall also discuss some results for $p \neq 2$. 

In this section, we focus on the problem of isotonic regression (Example~\ref{ex:IsoMdl}) and describe bounds on the risk of the isotonic LSE. As mentioned in the Introduction, isotonic regression is the most studied problem in shape-restricted regression where the risk behavior of the LSE is well-understood. We shall present the main results here. The results described in this section will serve as benchmarks to which risk bounds for other shape-restricted regression problems (see Section \ref{sec:Risk}) can be compared.   

Throughout this section, $\hat{\theta}$ will denote the isotonic LSE (which is the minimizer of $\|Y - \theta\|^2$ subject to the constraint that $\theta$ lies in the closed convex cone $\I$ described in~\eqref{eq:IsoCone}) and $\theta^*$ will usually denote an arbitrary vector in $\I$ (in some situations we deal with misspecified risks where $\theta^*$ is an arbitrary vector in $\R^n$ not necessarily in $\I$). 

The risk, $R(\hat{\theta}, \theta^*)$, essentially has two different kinds of behavior. As long as $\theta^* \in \I$ and $V(\theta^*) := \theta_n^* - \theta^*_1$ (referred to as the \textit{variation} of $\theta^*$) is bounded from above independently of $n$, the risk $R(\hat{\theta}, \theta^*)$ is bounded from above by a constant multiple of $n^{-2/3}$. We shall refer to this $n^{-2/3}$ bound as the \textit{worst case} risk bound mainly because it is, in some sense, the maximum possible rate at which $R(\hat{\theta}, \theta^*)$ converges to zero. On the other hand, if $\theta^* \in \I$ is \textit{piecewise constant} with not too many constant pieces, then the risk $R(\hat{\theta}, \theta^*)$ is bounded from above by the parametric rate $1/n$ up to a logarithmic multiplicative factor. This rate is obviously much faster compared to the worst case rate of $n^{-2/3}$ which means that the isotonic LSE is estimating piecewise constant nondecreasing sequences at a much faster rate. In other words, the isotonic LSE is \textit{adapting} to piecewise constant nondecreasing sequences with not too many constant pieces. We shall therefore refer to this $\log n/n$ risk bound as the \textit{adaptive} risk bound. 

The worst case risk bounds for the isotonic LSE will be explored in Section~\ref{sec:WC} while the adaptive risk bounds are treated in Section~\ref{sec:AdapRB}. Proofs will be provided in the Appendix~\ref{sec:Proofs}. Before proceeding to risk bounds, let us first describe some basic properties of the isotonic LSE.  
% %We closely follow~\cite{Bellec15}.
%\subsection{Risk bounds in isotonic regression}

An important fact about the isotonic LSE is that $\hat \theta = (\hat{\theta}_1, \dots, \hat{\theta}_n)$ can be  explicitly represented as (see~\cite[Chapter 1]{RWD88}):  
\begin{equation}\label{eq:IsoCharac}
 \hat{\theta}_j  = \min_{v \geq j} \max_{u \leq j} \frac{\sum_{l=u}^v y_l}{v-u+1}, \quad \mbox{for } j = 1,\ldots, n.
\end{equation} 
This is often referred to as the min-max formula for isotonic regression. The isotonic LSE is, in some sense, unique among shape-restricted regression LSEs because it has the above explicit characterization. It is this characterization that allows for a precise study of the properties of $\hat{\theta}$. 

The above characterization of the  isotonic LSE shows that $\hat \theta$ is {\it piecewise constant}, and in each ``block'' (i.e., region of constancy) it is the average of the response values (within the block); see~\cite[Chapter 1]{RWD88}. However, the blocks, their lengths and their positions, are chosen adaptively by the algorithm, the least squares procedure. If $\theta^*_i = f(x_i)$ for some design points $0 \leq x_1 < \dots < x_n \le 1$, then we can define the isotonic LSE of $f$ as the piecewise constant function $\hat f_n:[0,1] \to \R$ which has jumps only at the design points and such that $\hat f_n(x_i) = \hat \theta_i$ for each $i = 1,\ldots, n$. Figure~\ref{fig:IsoReg} shows three different scatter plots, for three different regression functions $f$, with the fitted isotonic LSEs $\hat f_n$. Observe that for the leftmost plot the block-sizes (of the isotonic LSE) vary considerably with the change in slope of the underlying function $f$ --- the isotonic LSE, $\hat{f}_n$, is nearly constant in the interval $[0.3, 0.7]$ where $f$ is relatively flat whereas $\hat{f}_n$ has many small blocks towards the boundary of the covariate domain where $f$ has large slope. This highlights the adaptive nature of the isotonic LSE $\hat f_n$ and also provides some intuition as to why the isotonic LSE adapts to piecewise constant nondecreasing functions with not too many constant pieces. Moreover, in some sense, $\hat f_n$ can be thought of as a kernel estimator (with the box kernel) or a `regressogram' (\cite{Tukey61}), but with a varying bandwidth/window.

\begin{figure}
\includegraphics[height = 2.0in, width = 5.2in]{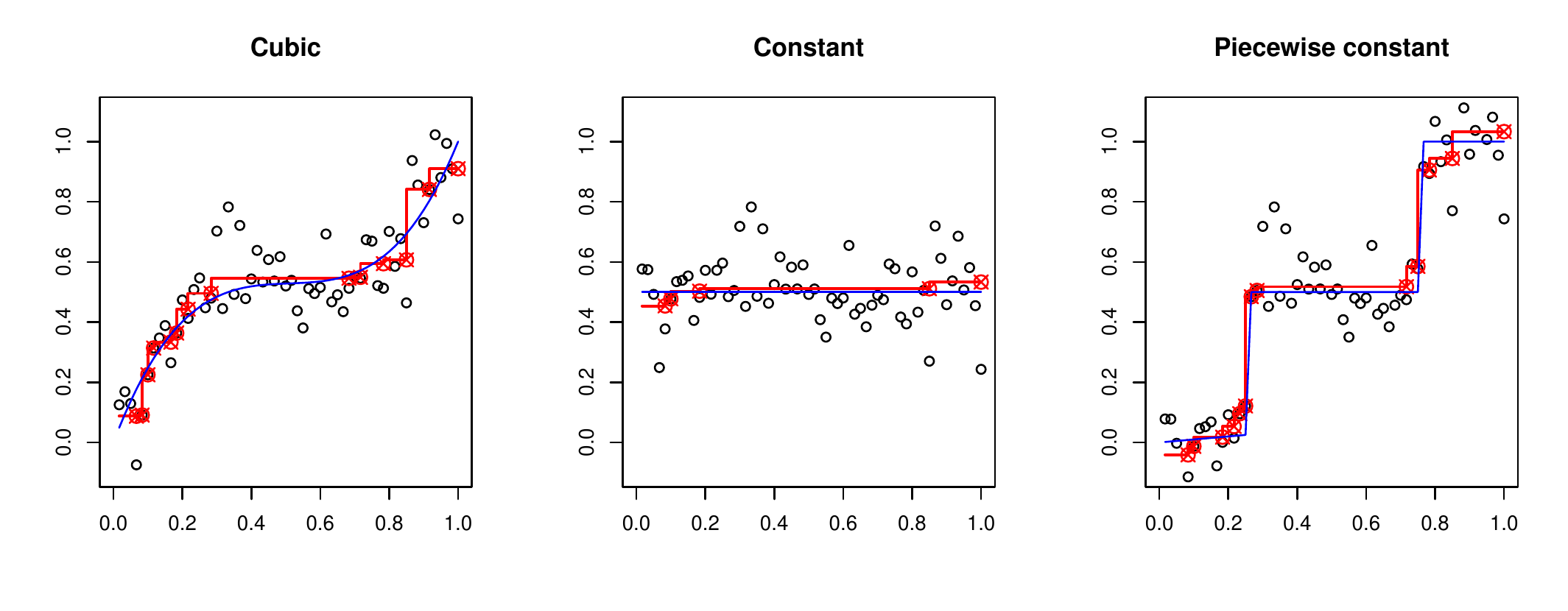}
\caption{Plots of $Y$ (circles), $\hat \theta$ (red), and $\theta^*$ (blue) for three different choices of $f$: (i) cubic polynomial (left plot), (ii) constant (middle plot), and (iii) piecewise constant. Here $n = 60$, and $\eps \sim N_n(0,\sigma^2 I_n)$ with $\sigma = 0.1$. Here $I_n$ denotes the identity matrix of order $n$.} 
\label{fig:IsoReg}
\end{figure}

\subsection{Worst case risk bound}\label{sec:WC}
The worst case risk bound for the isotonic LSE is given by the following inequality. Under the assumption  that the errors $\eps_1, \dots, \eps_n$ are i.i.d.~with mean zero and variance $\sigma^2$, the risk of the 
isotonic LSE satisfies the bound (see~\cite{Zhang02}):  
\begin{equation}\label{eq:IsoWC}
  R(\hat{\theta}, \theta^*) \le C \left(\frac{\sigma^2 V(\theta^*)}{n}
  \right)^{2/3} + C \frac{\sigma^2 \log (en)}{n}  
\end{equation}
where $V(\theta^*) = \theta_n^*  - \theta_1^*$ denotes the {\it variation} of $\theta^* \in \I$ and $C >0$ is a universal constant. 

Let us try to understand each of the terms on the right side of~\eqref{eq:IsoWC}. As long as the variation  $V(\theta^*)$ is not small the risk $R(\hat{\theta}, \theta^*)$ is given by $(\sigma^2 V(\theta^*)/n)^{2/3}$, up to a constant multiplicative factor. This shows that the rate of estimating any monotone function (under the $\ell_2$-loss) is $n^{-2/3}$. Moreover,~\eqref{eq:IsoWC} gives the explicit dependence of the risk on the variation of $\theta^*$ (and on $\sigma^2$). 
 
 The second term on the right side of~\eqref{eq:IsoWC} is also interesting --- when $V(\theta^*) = 0$, i.e., $\theta^*$ is a constant sequence,~\eqref{eq:IsoWC} shows that the risk of the isotonic LSE scales like $\log n/n$. This is a consequence of the fact that $\hat \theta$ chooses its blocks (of constancy) adaptively depending on the data. When $\theta^*$ is the constant sequence, $\hat \theta$ has fewer blocks (in fact, it has of the order of $\log n$ blocks; see \cite[Theorem 3]{bellec2016adaptive} and \cite[Theorem 1]{MW00}) and some of the blocks will be very  large (see e.g., the middle plot of Figure~\ref{fig:IsoReg}), so that averaging the responses within the large blocks would yield a value very close to the grand mean $\bar Y = (\sum_{i=1}^n y_i)/n$ (which has risk $\sigma^2/n$ in this problem). Thus~\eqref{eq:IsoWC} already illustrates the {\it adaptive} nature of the LSE --- the risk of the LSE $\hat \theta$ changes depending on the ``structure'' of the true $\theta^*$. In the next subsection (see~\eqref{eq:IsoAdapBd}) we further highlight this adaptive nature of the LSE.

%Some remarks on this bound. 

%. When $V(\theta^*)$ does not satisfy this condition, the risk $R(\hat{\theta}, \theta^*)$  becomes $\sigma^2 (\log (en)/n)$. Note that this is the same rate that we get for $V(\theta^*) = 0$ from the adaptive risk bound,  \eqref{eq:IsoAdapBd}.

\begin{remark}
To the best of our knowledge, inequality \eqref{eq:IsoWC} first appeared in~\cite[Theorem 1]{MW00} who proved it under the assumption that the errors $\eps_1, \dots, \eps_n$ are i.i.d.~$N(0, \sigma^2)$. Zhang~\cite{Zhang02} proved~\eqref{eq:IsoWC} for much more general  errors including the case when $\eps_1, \dots, \eps_n$ are i.i.d.~with mean zero and variance $\sigma^2$. The proof we give (in Section~\ref{sec:IsoReg-WC} of Appendix~\ref{sec:Proofs}) follows the  arguments of \cite{Zhang02}. Another proof of an inequality similar to~\eqref{eq:IsoWC} for the case of normal errors has been given recently by \cite{Chat14} who proved it as an illustration of a general technique for bounding the risk of LSEs.
\end{remark}

\begin{remark} 
The LSE over bounded monotone functions also satisfies the bound~\eqref{eq:IsoWC} and has been observed by many authors including~\cite{nemirovski1985convergence, VandeGeer90annstat, Donoho90}. Proving this result is easier, however, because of the presence of the uniform bound on the function class (such a bound is not present for the isotonic LSE). It must also be kept in mind that the bounded isotonic LSE comes with a tuning parameter that needs to chosen by the user. 
\end{remark}

\begin{remark}\label{reim}
    Inequality \eqref{eq:IsoWC} also implies that the isotonic LSE achieves the risk $(\sigma^2 V/n)^{2/3}$ for $\theta^* \in \I_V := \{\theta \in \I : \theta_n - \theta_1 \leq V\}$ (as long as $V$ is not too small) without any knowledge of $V$. It turns out that the minimax risk over $\I_V$  is of the order $(\sigma^2 V/n)^{2/3}$ as long as $V$  is in the range $\sigma/\sqrt{n} \lesssim V \lesssim \sigma n$ (see e.g., \cite[Theorem 5.3]{CGS15}). Therefore, in this wide range of $V$, the isotonic LSE is minimax (up to constant multiplicative factors) over the class $\I_V$. This is especially interesting because the isotonic LSE does not require any knowledge of $V$. This illustrates another kind of adaptation of the isotonic LSE; further details on this can be found in \cite{CL17}. 
\end{remark}

%The behavior of $\hat{\theta}$, under the $\ell_2$-loss, has been studied in a number of papers including \cite{vdG90, vdG93, Donoho90, BM93, Wang96, MW00, Zhang02}. Two types of risk bounds have been obtained so far.~\cite[Theorem 2.2]{Zhang02} showed 

\subsection{Adaptive risk bounds}\label{sec:AdapRB}
As the isotonic LSE fit is piecewise constant, it may be reasonable to expect that when $\theta^*$ is itself a piecewise constant (with not too many pieces), the risk of $\hat \theta$ would be small. The rightmost plot of Figure~\ref{fig:IsoReg} corroborates this intuition. This leads us to our second type of risk bound for the LSE. For $\theta \in \I$, let $k(\theta) \ge 1$ denote the number of constant blocks of $\theta$, i.e., $k(\theta)$ is the integer such that $k(\theta) -1$ is the number of inequalities $\theta_i \le \theta_{i+1}$ that are strict, for $i = 1,\ldots, n-1$ (the number of jumps of $\theta$). 
\begin{theorem}\label{thm:AdapIsoReg}
Under the assumption that $\eps_1, \dots,  \eps_n$ are i.i.d.~with mean zero and variance $\sigma^2$ we have 
\begin{equation}\label{eq:IsoAdapBd}
  R(\hat{\theta}, \theta^*) \leq \inf_{\theta \in \I}
  \left[\frac{1}{n} \|\theta^* - \theta\|^2 + \frac{4\sigma^2
      k(\theta)}{n} \log \frac{en}{k(\theta)} \right]
\end{equation}
for every $\theta^* \in \R^n$. 
\end{theorem}
%{\color{red} What about 4 in the Theorem}
Note that $\theta^*$ in Theorem \ref{thm:AdapIsoReg} can be any arbitrary vector in $\R^n$ (it is not required that $\theta^* \in \I$). An important special case of inequality~\eqref{eq:IsoAdapBd} arises when $\theta^* \in \I$ and $\theta$ is taken to be $\theta^*$ in order to obtain: 
\begin{equation}\label{mada}
  R(\hat \theta, \theta^*) \le \frac{4\sigma^2 k(\theta^*)}{n} \log \frac{en}{k(\theta^*)}. 
\end{equation}
It makes sense to compare~\eqref{mada} with the worst case risk bound \eqref{eq:IsoWC}. Suppose, for example,  $\theta_j^* = \mathbf{1}\{j > n/2\}$ (here $\mathbf{1}$ denotes the indicator function) so that $k(\theta^*) = 2$ and $V(\theta^*) = 1$. Then the risk bound in~\eqref{eq:IsoWC} is essentially $(\sigma^2/n)^{2/3}$ while the right side of~\eqref{mada} is $(8\sigma^2/n) \log (en/2)$ which is much smaller than $(\sigma^2/n)^{2/3}$. More generally, if $\theta^*$ is piecewise constant with $k$ blocks then $k(\theta^*) = k$ so that inequality~\eqref{mada} implies that the risk is given by the parametric rate $k \sigma^2/n$ with a logarithmic multiplicative factor of $4 \log(en/k)$ --- this is a much stronger bound compared to \eqref{eq:IsoWC} when $k$ is small. 

Inequality \eqref{mada} is an example of an \textit{oracle inequality}. This is because of the following. Let $\hat{\theta}^{OR}$ denote the oracle piecewise constant estimator of $\theta^*$ which estimates $\theta^*$ by the mean of $Y$ in each constant block of $\theta^*$ (note that $\hat{\theta}^{OR}$ uses knowledge of the locations of the constant blocks of $\theta^*$ and hence is an oracle estimator). It is easy to see then that the risk of $\hat{\theta}^{OR}$ is given by 
$$R(\hat \theta^{OR}, \theta^*) = \frac{\sigma^2 k(\theta^*)}{n}.$$
As a result, inequality \eqref{mada} can be rewritten as 
\begin{equation}\label{mora}
  R(\hat{\theta}, \theta^*) \leq \left(4 \log \frac{en}{k(\theta^*)} \right) R(\hat{\theta}^{OR}, \theta^*). 
\end{equation}
Because this involves a comparison of the risk of the LSE $\hat{\theta}$ with that of the oracle  estimator $\hat{\theta}^{OR}$, inequality~\eqref{mada} is referred to as an oracle inequality. Inequality \eqref{mora} shows that the isotonic LSE, which uses no knowledge of $k(\theta^*)$ and the positions of the blocks, has essentially the same risk performance as the oracle piecewise constant  estimator (up to the multiplicative logarithmic factor $4 \log (en/k(\theta^*))$). This is indeed remarkable! 

For certain piecewise constant vectors $\theta^*$ with $k$ blocks, it might be possible to approximate $\theta^*$ closely with another piecewise constant vector $\tilde{\theta}$ having $k'$ blocks where $k' < k$. In such cases, it makes sense to compare the performance of the isotonic estimator $\hat{\theta}$ to the oracle piecewise constant estimator with $k'$ blocks. Such a comparison is achieved by inequality \eqref{eq:IsoAdapBd} which is a stronger inequality than \eqref{mada}. In fact, \eqref{eq:IsoAdapBd} can  actually be viewed as a more general oracle inequality where the behavior of the isotonic LSE is compared with oracle piecewise constant estimators even when $\theta^* \notin \I$. We would like to mention here that, in this context, \eqref{eq:IsoAdapBd} is referred to as a \textit{sharp oracle inequality} because the leading constant in front of the $\norm{\theta^* - \theta}^2/n$ term on the right-hand side of \eqref{eq:IsoAdapBd} is equal to one. We refer to \cite{Bellec15} for a detailed  explanation of oracle and sharp oracle inequalities. 

Based on the discussion above, it should be clear to the reader that the adaptive risk bound~\eqref{eq:IsoAdapBd} complements the worst case bound~\eqref{eq:IsoWC} as it  gives much finer information about how well any particular $\theta^*$ (depending on its `complexity') can be estimated by the LSE $\hat \theta$.

\begin{remark}[Model misspecification]\label{momi}
As already mentioned, the sharp oracle inequality~\eqref{eq:IsoAdapBd} needs no assumption on $\theta^*$  (which can be any arbitrary vector in $\R^n$), i.e., the inequality holds true even when $\theta^* \notin \I$. See~\cite[Section 6]{CGS15} for another way of handling model misspecification, where $\hat \theta$ is compared with the ``closest'' element to $\theta^*$ in $\I$ (and not $\theta^*$).
\end{remark}

\begin{remark}
To the best of our knowledge, an inequality of the form \eqref{eq:IsoAdapBd} first explicitly appeared in \cite[Theorem 3.1]{CGS15} where it was proved that 
\begin{equation}\label{wea}
  R(\hat{\theta}, \theta^*) \leq 4 \inf_{\theta \in \I}
  \left[\frac{1}{n} \|\theta^* - \theta\|^2 + \frac{4\sigma^2
      k(\theta)}{n} \log \frac{en}{k(\theta)} \right] 
\end{equation}
 under the additional assumption that $\theta^* \in \I$. The proof of this inequality given in \cite{CGS15} is based on ideas developed in \cite{Zhang02}. Note the additional constant factor of $4$ in the above inequality compared to \eqref{eq:IsoAdapBd}. 

Under the stronger assumption $\eps \sim N_n(0, \sigma^2 I_n)$, Bellec~\cite[Theorem 3.2]{Bellec15} improved \eqref{wea} and proved that 
\begin{equation}\label{eq:IsoAdapBd2}
R(\hat \theta, \theta^*) \le \inf_{\theta \in \I} \left[\frac{1}{n} \|\theta^* - \theta\|^2 + \frac{\sigma^2 k(\theta)}{n} \log \frac{e n}{k(\theta)} \right],
\end{equation}
for every $\theta^* \in \R^n$. A sketch of the proof of this inequality is given in Subsection \ref{Bellec-Adap} of Appendix~\ref{sec:Proofs}. A remarkable feature of this bound is that the multiplicative constants involved are  all tight, which implies, in particular, that 
\begin{equation*}
  R(\hat{\theta}, \theta^*) \leq \inf_{\theta \in \I}
  \left[\frac{1}{n} \|\theta^* - \theta\|^2 + C\frac{\sigma^2
      k(\theta)}{n} \log \frac{en}{k(\theta)} \right]
\end{equation*}
cannot hold for every $\theta^*$ if $C < 1$. This follows from the fact that when $\theta^* = (0, 0, \dots, 0)  \in \I$ and $\eps \sim N_n(0, \sigma^2 I_n)$, the risk $R(\hat{\theta}, \theta^*)$ exactly equals $\sigma^2 \sum_{j=1}^n   1/j \asymp \sigma^2 \log n$; see~\cite{Bellec15} for an explanation. It must be noted that this implies, in particular, that the logarithmic term in these adaptive risk bounds cannot be removed. 

Note that inequality \eqref{eq:IsoAdapBd} has an additional factor of $4$ compared to \eqref{eq:IsoAdapBd2}   on the second term in the right-hand side. This is because the errors $\eps_1, \dots, \eps_n$ can be non-Gaussian in Theorem \ref{thm:AdapIsoReg}. 
\end{remark}
\begin{remark}
One may attempt to prove \eqref{eq:IsoWC}  from the adaptive
  risk bound \eqref{eq:IsoAdapBd} by approximating arbitrary $\theta^*
  \in \I$ via $\theta \in \I$ with a bound on $k(\theta)$. However, it
  is likely that such an approach will lead to additional logarithmic
  terms on the right hand side of \eqref{eq:IsoWC} (see e.g.,~\cite[Theorem 4.1]{CGS15}). 
\end{remark}

\begin{remark}\label{beli}
  For some choices of $\theta^* \in \M$, it is possible to obtain bounds on the risk $R(\hat{\theta}, \theta^*)$ of the LSE which combine aspects of both \eqref{eq:IsoWC} and \eqref{eq:IsoAdapBd}. For example, if $\theta^*$ is piecewise constant with $k$ blocks for $1 \leq i \leq n/2$ and if it is strictly increasing with variation bounded by $V$ for $n/2 \leq i \leq n$, then it can be shown that the risk of the LSE will be bounded from above by a constant multiple of $\sigma^2 (k/n) \log ({e n}/{k})  + (\sigma^2 V/n)^{2/3}$. Techniques for obtaining such hybrid risk bounds in isotonic regression can be found in \cite[Sections 2 and 3]{Zhang02}. 
\end{remark}
%Note that inequality~\eqref{eq:IsoAdapBd} holds under the assumption that $\eps_1,\ldots, \eps_n$ are i.i.d.~$N(0,\sigma^2)$. The following (slightly weaker) inequality can be deduced when $\epsilon_1, \dots,  \epsilon_n$ are i.i.d.~with mean zero and variance $\sigma^2$ but not  necessarily normal. 

%{\color{red} Till now we have highlighted the attractive features of the isotonic LSE, but have not mentioned anything about how other estimators might perform in this problem. This is usually tackled in statistics by computing the minimax rate of convergence for the problem. Indeed, it is known that $\hat \theta$ is minimax optimal; in fact, $\hat \theta$ is ``locally'' minimax optimal (in a non-asymptotic sense); see~\cite[Section 5]{CGS15}. 

%It was showed in~\cite{Cator2011} that the rate of convergence of $\hat{f}_{ls}(x_0)$ to $f_0(x_0)$ depends on the local behavior of $f_0$ near $x_0$ and explicitly described this rate for each $f_0$. In this sense, the LSE $\hat{f}_{ls}$ adapts  automatically to the unknown function $f_0$. In \cite{Cator2011}, it was also proved that the LSE is optimal for local behavior by establishing a local asymptotic minimax lower bound.  }

\subsubsection{Adaptive Risk Bounds for $R^{(p)}(\hat{\theta}, \theta^*)$.} \label{pp}
The risk bound \eqref{eq:IsoAdapBd} (or more specifically \eqref{mora}) implies that the isotonic LSE pays a logarithmic price in risk compared to the oracle piecewise constant estimator. This fact is strongly tied to the fact that the risk is measured via squared error loss (as in \eqref{eq:Risk}). The story will be different if  one measures risk under $\ell_p$-metrics for $p \neq 2$. To illustrate this, we shall describe adaptive bounds for the risk $R^{(p)}(\hat{\theta}, \theta^*)$  defined in \eqref{rlp}.

The following result bounds the risk $R^{(p)}(\hat{\theta}, \theta^*)$
assuming that $\theta^* \in \I$. The risk bounds involve a positive
constant $C_p$ that depends on $p$ alone. Explicit expressions for
$C_p$ can be gleaned from the proof of Theorem \ref{thm:L_p} (in Section~\ref{sec:R_p-Risk}).  

\begin{theorem}\label{thm:L_p}
Assume that the errors $\eps_1, \dots, \eps_n$ are i.i.d.~$N(0, \sigma^2)$. Fix $\theta^* \in \I$ and let $p \geq 1, p \neq 2$. Let $k$
denote the 
number of constant blocks of $\theta^*$ and let the lengths of the
blocks be denoted by $n_1, \dots, n_k$. We then have  
\begin{equation}\label{pab.eq}
  R^{(p)}(\hat{\theta}, \theta^*) \leq C_p \frac{\sigma^p}{n}
  \sum_{i=1}^k n_i^{(2 - p)_+/2} \leq C_p \sigma^p \left(\frac{k}{n}
  \right)^{\min(p, 2)/2}
\end{equation}
where $C_p$ is a positive constant that depends on $p$ alone.   
\end{theorem}

\begin{remark}
As stated, Theorem \ref{thm:L_p} appears to be new even though its conclusion is implicit in the detailed risk calculations of~\cite{Zhang02} for isotonic regression. We have assumed that $\eps_1, \dots, \eps_n$ are normal in Theorem \ref{thm:L_p} but it is possible to allow non-Gaussian errors by  imposing suitable moment conditions. 
\end{remark}

\begin{remark}\label{belia}
  From an examination of the proof of Theorem \ref{thm:L_p} (given in Subsection \ref{sec:R_p-Risk}), it is evident that the constant $C_p$ tends to  $+\infty$ as $p \rightarrow 2$. Note that this makes sense because when $p = 2$, the right hand side of \eqref{pab.eq} equals $C_p \sigma^2 k/n$ and we know from the previous subsection that there must be a logarithmic term (in $n$) for the risk when $p = 2$. It is helpful here to note that by Jensen's inequality (and the bound \eqref{mada}), we have, for $1 \leq p \le 2$, the bound
  \begin{equation*}
    R^{(p)}(\hat{\theta}, \theta^*) \leq \left(R^{(2)}(\hat{\theta}, \theta^*) \right)^{p/2} \leq 4^{p/2} \sigma^p \left(\frac{k}{n} \right)^{p/2} \left(\log \frac{en}{k} \right)^{p/2}. 
  \end{equation*}
which does not explode as $p \uparrow 2$. The above bound can also be obtained by modifying the proof of Theorem \ref{thm:L_p} where in place of the inequality 
\begin{equation}\label{cu1}
  \sum_{j=1}^n j^{-p/2} \leq \frac{2}{2 - p} n^{1- (p/2)} \qt{for $1
    \leq p < 2$},
\end{equation}
we use 
\begin{equation}\label{cu2}
  \sum_{j=1}^n j^{-p/2} \leq n \left(\frac{1}{n} \sum_{j=1}^n
    \frac{1}{j} \right)^{p/2} \le   n \left(\frac{\log (en)}{n}
  \right)^{p/2} \qt{for all $1 \leq p \leq 2$}
\end{equation}
which is again a consequence of Jensen's inequality. 
\end{remark}

Let us now compare the isotonic LSE to the oracle piecewise constant estimator $\hat{\theta}^{OR}$ (introduced in the previous subsection) in terms of the $\ell_p$-risk. It is easy to verify that the risk of $\hat{\theta}^{OR}$ under the $\ell_p$-loss is given by 
  \begin{equation}\label{orp}
    R^{(p)}(\hat{\theta}^{OR}, \theta^*) = (\E |\eta|^p) \sigma^p \frac{1}{n}
    \sum_{i=1}^k n_i^{(2 - p)/2} 
  \end{equation}
for every $p > 0$ where $\eta := \eps_1/\sigma$ is standard normal. 
  
Comparing \eqref{pab.eq} and \eqref{orp}, we see that the isotonic LSE performs at the same rate (up to constant multiplicative factors) as the oracle piecewise constant estimator for $1 \leq p < 2$ (there is not even a logarithmic price for these values of $p$). When $p = 2$, as seen from \eqref{mora}, the isotonic LSE  pays a logarithmic price of $4 \log (en/k(\theta^*))$. For $p > 2$ however, there is a significant price that is paid. For example, if all the constant blocks have roughly equal size, then the oracle estimator's risk, when $p > 2$, is of order $(k/n)^{p/2}$ while the bound in~\eqref{pab.eq} is of order $k/n$. It is also actually true that if $\I_k$ denotes the class of all $\theta^* \in \I$ with $k$ constant blocks, then (for a positive constant $C_p$) 
\begin{equation}\label{rlb}
  \sup_{\theta^* \in \I_k} R^{(p)}(\hat{\theta}, \theta^*) \geq C_p \sigma^p \left(\frac{k}{n} \right) \qt{for every $p > 2$}
\end{equation}
and this confirms the fact that there is a significant price to be paid by the LSE (compared to the oracle piecewise constant estimator) for estimating $\theta^* \in \I_k$ when the risk is measured by $R^{(p)}(\hat{\theta}, \theta^*)$ for $p > 2$. A sketch of the proof of \eqref{rlb} is given in Subsection \ref{prlb}. 

Theorem \ref{thm:L_p} can be generalized to situations where $\theta^* \in \I$ has a large number of constant blocks provided it can be well-approximated by $\theta \in \I$ with a small (compared to $n$) number of constant blocks. This result is given below (and proved in Subsection \ref{ppro}). It is  similar in spirit to \eqref{eq:IsoAdapBd} even though it is not as sharp or clean as \eqref{eq:IsoAdapBd}. We need some notation to state this result. An interval partition $\pi$ of $n$ is a finite sequence of positive integers that sum to $n$. Let $\Pi$ denote the set of all such interval partitions $\pi$ of $n$. For each $\pi = (n_1, \dots, n_k) \in \Pi$, let $k(\pi) := k$. The variation of $\theta \in \I$ with respect to $\pi \in \Pi$ is defined as
\begin{equation*}
  V_{\pi}(\theta) := \max_{1 \leq i \leq k} \left(\theta_{s_i} - \theta_{s_{i-1}} \right) 
\end{equation*}
where $s_i, 0 \leq i \leq k$ are defined (with respect to the partition $\pi := (n_1, \dots, n_k)$) as $s_0  := 0$ and $s_i := n_1 + \dots + n_i$ for $i = 1, \dots, k$. 
\begin{theorem}\label{pro}
  Assume that the errors $\epsilon_1, \dots, \epsilon_n$ are i.i.d $N(0, \sigma^2)$. Fix $\theta^* \in \I$ and let $p \geq 1, p \neq 2$. Then 
  \begin{equation}\label{pro.eq}
    R^{(p)}(\hat{\theta}, \theta^*) \leq C_p \inf_{\pi \in \Pi}\left( \left[V_{\pi}(\theta^*)\right]^p + \sigma^p  \left(\frac{k(\pi)}{n}  \right)^{\min(p, 2)/2} \right)
  \end{equation}
  for a positive constant $C_p$ that depends on $p$ alone. 
\end{theorem}
Unlike \eqref{eq:IsoAdapBd}, inequality \eqref{pro.eq} is not a sharp oracle inequality because it only holds for $\theta^* \in \I$ (and not for general $\theta^* \in \R^n$) and because the constant in front  of the $V_{\pi}(\theta^*)$ term is not one. However it is still useful and it includes Theorem \ref{thm:L_p} as a special case (indeed to derive \eqref{pab.eq} from \eqref{pro.eq}, just take the partition $\pi$ which corresponds to the constant blocks of $\theta^*$). The bound \eqref{pro.eq} can  also be used to obtain worst case risk bounds for the LSE in terms of the $L^p$ risk for $1 \leq p < 2$ (analogous to \eqref{eq:IsoWC}).  Indeed, it can be shown (see, for example, \cite[Lemma 11.1 in the supplementary material]{CGS15}) that for every $\theta^* \in \I$ and $\delta > 0$, there exists $\pi \in \Pi$ with 
\begin{equation*}
  V_{\pi}(\theta^*) \leq \delta ~~ \text{ and } ~~ k(\pi) \leq 1 + \frac{V(\theta^*)}{\delta}.
\end{equation*}
This implies from \eqref{pro.eq} that 
\begin{equation*}
  R^{(p)}(\hat{\theta}, \theta^*) \leq C_p \inf_{\delta > 0}  \left(\delta^p + \sigma^p \left(\frac{1}{n} + \frac{V(\theta^*)}{n\delta} \right)^{\min(p, 2)/2}  \right). 
\end{equation*}
From here, it can be shown that
\begin{equation*}
  R^{(p)}(\hat{\theta}, \theta^*) \leq C_p \left(\frac{\sigma^2 V(\theta^*)}{n} \right)^{p/3} + C_p \left(\frac{\sigma^2}{n} \right)^{p/2} \qt{for $1 \leq p < 2$}. 
\end{equation*}
It turns out that this bound cannot be improved (up to the multiplicative factor $C_p$) as argued in \cite[Theorem 2.2 and the following discussion]{Zhang02}. We would like to remark here that this method will lead to a suboptimal worst case risk bound for $R^{(p)}(\hat{\theta}, \theta^*)$ for $p > 2$.

\section{Risk bounds in other shape-restricted regression problems}\label{sec:Risk}
In this section we consider the problems of convex regression (Example \ref{ex:ConvexReg}), isotonic regression on a partially ordered set (Example \ref{ex:RegPartialOrder}), unimodal regression (Example \ref{ex:UniReg}) and shape restricted additive models (Example \ref{ex:AddMdl}). In each of these problems, we describe results related to the performance of the LSEs. The reader will notice that the risk results are not as detailed as compared to the  isotonic regression results of the previous section. 

\subsection{Convex Regression}\label{cvre}
Let us consider Example~\ref{ex:ConvexReg} where the goal is to estimate a convex function $f: [0,1] \to \R$ from regression data as in~\eqref{eq:RegMdl}. The convex LSE $\hat \theta$ is defined as the projection of $Y$ onto the closed convex cone $\K$ (see~\eqref{eq:CvxReg}). This estimator was first proposed in~\cite{Hildreth54} for the estimation of production functions and Engel curves. It can be shown that $\hat \theta$ is piecewise affine with knots only at the design points; see~\cite[Lemma 2.6]{GJW01-a}. The accuracy of the LSE, in terms of the risk $R(\hat{\theta}, \theta^*)$ (defined in \eqref{eq:Risk}), was first studied in~\cite{GS15} followed by \cite{CGS15, Bellec15, Chat16}. These results are summarized below. Earlier results on the risk under a supremum loss can be found in~\cite{HanPled76, DuembgenEtAl04}. 

Suppose that $\eps \sim N_n(0, \sigma^2 I_n)$. In~\cite{Chat16}, the following worst case risk bound for   $\hat \theta$ was given (when $x_i = i/n$ are the ordered design points):
\begin{equation}\label{eq:CvxLSE-WC}
R(\hat \theta, \theta^*) \le C \left(\frac{\sigma^2 \sqrt{T(\theta^*)}}{n} \right)^{4/5} + C \frac{\sigma^2}{n^{4/5}}
\end{equation} 
where $C>0$ is a universal constant and $T(\theta^*)$ is a constant depending on $\theta^*$ (like $V(\theta^*)$ in~\eqref{eq:IsoWC} for isotonic regression). Roughly speaking, $T(\theta^*)$ measures the ``distance'' of $\theta^*$ from the set of all affine (functions) sequences. Formally, Let $L$ denote the subspace of $\R^n$ spanned by the constant vector $(1, \ldots,1)$ and the vector $(1, 2,\ldots, n)$;  i.e., $L$ is the linear subspace of affine sequences. Let $P_L$
denote the orthogonal projection matrix onto the subspace $L$ and let $\beta^* := (I_n - P_L) \theta^*$. Then $T(\theta^*) := \max_{1\le i \le n} \beta_i^* - \min_{1\le i \le n} \beta_i^*$. Observe that when $\theta^*$ itself is an affine sequence (which is also a convex sequence), then $T(\theta^*) = 0$. 

The risk bound~\eqref{eq:CvxLSE-WC} shows that the risk of the convex LSE is bounded above by $n^{-4/5}$. Inequality~\eqref{eq:CvxLSE-WC} improved a result in~\cite{GS15}, which had a similar bound but with an additional multiplicative logarithmic factor (in $n$). Comparing with \eqref{eq:IsoWC}, it is natural to conjecture that the second term in \eqref{eq:CvxLSE-WC} can be improved to $C \sigma^2 (\log (en))/n$ but this has not been proved so far. Another feature of \eqref{eq:CvxLSE-WC} is that the errors are assumed to be Gaussian; it might be possible to extend them to sub-Gaussian errors but this is still a strong assumption compared to the corresponding result for isotonic regression (see~\eqref{eq:IsoWC}) which holds without distributional assumptions. 

The proof of \eqref{eq:CvxLSE-WC} (and other worst case risk bounds like~\eqref{eq:CvxLSE-WC} for shape-restricted regression problems under Gaussian/sub-Gaussian errors) involves tools from the theory of Gaussian processes like chaining and Dudley's entropy bound and crucially relies on an accurate `size' measure of the underlying class (e.g., `local' balls of $\K$) as captured by its metric entropy; see Section~\ref{sec:Th-LSE} for a broad outline of the proof strategy. Although the main idea of the proof is simple, deriving appropriate bounds on the metric entropy of the underlying  class can be challenging. 

As with the isotonic LSE, the convex LSE $\hat \theta$ exhibits adaptive behavior. As the convex LSE $\hat \theta$ is piecewise affine it may be expected that the risk of $\hat \theta$ would be nearly parametric if the true $\theta^*$ is (well approximated by) a piecewise affine function. Indeed this is the case. For $\theta \in \K$ let $q(\theta) \ge 1$ denote the number of affine pieces of $\theta$; i.e., $q(\theta)$ is an integer such that $q(\theta) - 1$ is the number of inequalities in~\eqref{eq:CvxReg} that are strict. This adaptive behavior can be illustrated through the following risk bound:
\begin{equation}\label{yno}
R(\hat{\theta}, \theta^*) \leq \inf_{\theta \in \K} \left[\frac{1}{n} \|\theta^* - \theta\|^2 + \frac{8\sigma^2 q(\theta)}{n} \log \frac{en}{q(\theta)} \right].
\end{equation}
This inequality has been proved by \cite[Section 4]{Bellec15} improving earlier results of \cite{GS15, CGS15} which had superfluous multiplicative constants. %A sketch of the proof of this inequality is given in Section \ref{sec:IsoReg-AB}. 
Note that this bound holds for $\eps \sim N(0, \sigma^2 I_n)$. It is not known if the bounds holds for non-Gaussian errors (compare this with the corresponding inequality \eqref{eq:IsoAdapBd} for isotonic regression which holds without distributional assumptions on the errors).  Let us also note that risk bounds for the LSE under the $R^{(p)}(\hat{\theta}, \theta^*)$ risk (defined in \eqref{rlp}) are not available for convex regression.

\subsection{Isotonic regression on a partially ordered set}
We now turn our attention to Example~\ref{ex:RegPartialOrder} where the covariates are partially ordered and the goal is to estimate the order preserving (isotonic) regression function. The book Robertson et al.~\cite[Chapter 1]{RWD88} gives a nice overview of the characterization and computation of LSEs in such problems along with their applications in statistics. However, not much is known in terms of rates of convergence for these LSEs beyond the example of coordinate-wise nondecreasing ordering introduced in Example~\ref{ex:RegPartialOrder}. 

In this subsection we briefly review the main results in~\cite{CGS17} which considers estimation of a bivariate ($d = 2$) coordinate-wise nondecreasing regression function. An interesting recent paper~\cite{Han17} has extended these results to all dimensions $d \geq 2$ (see Remark \ref{hawa}). Estimation of bivariate coordinate-wise nondecreasing functions has applications and connections to the problem of estimating matrices of pairwise comparison probabilities arising from pairwise comparison data (\cite{CM17, ShahEtAl17}) and to seriation (\cite{Flam16}). 

%Also see~\cite{CL17, CM17, ShahEtAl17, Flam16} for various generalizations of isotonic regression beyond $d=1$ and a study of the risk properties of the corresponding LSEs. 

As the distribution of the design points $x_i$ complicate the analysis of shape-restricted LSEs, especially when $d >1$, for simplicity, we consider the regular uniform grid design. This reduces the problem to estimating an isotonic `matrix' $\theta^* := (\theta_{ij}^*) \in \R^{n_1 \times n_2}$ from observations   
\begin{equation*}\label{eq:RegMdl2}
{y}_{ij} = \theta_{ij}^* + \eps_{ij}, \qquad \mbox{for } i=1,\ldots, n_1, \; j = 1,\ldots, n_2, 
\end{equation*}
where $\theta^*$ is constrained to lie in
\begin{equation*}\label{eq:M}
  \M := \{ \theta \in \R^{n_1 \times n_2}: \theta_{ij} \leq \theta_{kl}
  \mbox{ whenever $i \leq k$ and $j \leq l$} \},  
\end{equation*}
and the random errors $\eps_{ij}$'s are i.i.d.~$N(0,\sigma^2)$,
with $\sigma^2>0$ unknown. We refer to any matrix in $\M$ as an isotonic matrix. Letting $Y := ({y}_{ij})$ denote the matrix (of order $n_1 \times n_2$; $n := n_1 n_2$) of the observed responses,  the LSE $\hat{\theta}$ is defined as the
minimizer of the squared Frobenius norm, $\|Y - \theta\|^2$, over $\theta \in \M$, i.e.,  
\begin{equation}\label{eq:MatLSE}
\hat{\theta} := \argmin_{\theta \in \M} \sum_{i = 1}^{n_1} \sum_{j = 1}^{n_2} ({y}_{ij} - \theta_{ij})^2.   
\end{equation}
As $\M$ is a closed convex cone in $\R^{n_1 \times n_2}$, the LSE $\hat \theta$ exists uniquely. 

The goal now is to formulate both the worst case and adaptive risk bounds for the matrix isotonic LSE $\hat \theta$ in estimating $\theta^*$. In~\cite[Theorem 2.1]{CGS17} it was shown that 
\begin{equation}\label{wwc}
R(\hat{\theta}, \theta^*) \leq C \left(\sqrt{\frac{\sigma^2 V^2(\theta^*)}{n}}
  (\log n)^4 + \frac{\sigma^2}{n} (\log n)^8 \right)   
\end{equation}
for a universal constant $C>0$, where $V(\theta^*) := \theta_{n_1n_2}^* - \theta_{11}^*$ is the variation of the isotonic matrix $\theta^*$. The above bound shows that when the variation $V(\theta^*)$ of $\theta^*$ is a non-zero constant, the risk of $\hat \theta$ decays at the rate $n^{-1/2}$, while when $V(\theta^*) = 0$ (i.e., $\theta^*$ is a constant), the risk is (almost) parametric. The above bound probably has superfluous logarithmic factors but bounds with smaller logarithmic factors have not yet been proved. Some understanding of the dependence of the bound \eqref{wwc} on $\sigma, V(\theta^*)$ and $n$ (which is different from the corresponding dependence in the one dimensional bound \eqref{eq:IsoWC}) can be derived from the following scaling argument. The risk $R(\hat{\theta}, \theta^*)$ only depends on $\theta^*$ and $\sigma$ so let us denote it by $g(\theta^*, \sigma)$. By a natural scaling argument (where we multiply all the observations by a constant $t > 0$), it should be clear that 
\begin{equation}\label{scaa}
  g(\theta^*, \sigma) = t^2 g(\theta^*/t, \sigma/t) \qt{for every $t > 0$}. 
\end{equation}
It is easy to see now that the same identity holds when $g(\theta^*, \sigma)$ is taken to be the right hand side of \eqref{wwc} as well. This will not be true if, for example, $V^2(\theta^*)$ is replaced by some other power of $V(\theta^*)$ in the right hand side of \eqref{wwc}. This argument, via the scaling identity \eqref{scaa}, can be used to understand the dependencies on $\theta^*$ and $n$ in the one-dimensional bound \eqref{eq:IsoWC} as well. 

To describe the adaptive risk bound for the matrix isotonic LSE we need to introduce some notation. A subset $A$ of $\{1, \dots, n_1\} \times \{1, \dots, n_2\}$ is called a {\it rectangle} if $A = \{(i, j) : k_1 \leq i \leq l_1, k_2 \leq j \leq l_2 \}$ for some $1 \leq k_1 \leq l_1 \leq n_1$ and $1 \leq k_2 \leq l_2 \leq n_2$. A rectangular partition of $\{1, \dots, n_1\} \times \{1, \dots, n_2\}$ is a collection of rectangles $\pi = (A_1, \dots, A_k)$ that are disjoint and whose union is $\{1, \dots, n_1\} \times \{1, \dots, n_2\}$. The cardinality of such a partition, $|\pi|$, is the number of rectangles in the partition. The collection of all rectangular partitions of $\{1, \dots, n_1\} \times \{1, \dots, n_2\}$ will be denoted by $\pp$. For $\theta \in \M$ and $\pi = (A_1, \dots, A_k) \in \pp$, we say that $\theta$ is constant on $\pi$ if
$\{\theta_{ij} : (i,j) \in A_l\}$ is a singleton for each $l = 1, \dots, k$. We define $k(\theta)$, for $\theta \in \M$, as the ``number of rectangular blocks'' of $\theta$, i.e., the smallest integer $k$ for which there exists a partition $\pi \in \pp$ with $|\pi| = k$ such that $\theta$ is constant on $\pi$. %It is trivial to see that $k(\theta) \geq c(\theta)$ for every $\theta \in \M$. As a simple illustration, for $\theta = \mathbf{1}\{i > 1, j > 1\},$ we have $c(\theta) = 2$ and $k(\theta) = 3.$   
In~\cite[Theorem 2.4]{CGS17} the following adaptive risk bound was stated: 
\begin{equation}\label{tama}
  R(\hat{\theta}, \theta^*) \leq  \inf_{\theta \in \M}  \left(\frac{\|\theta^* -
    \theta\|^2}{n} + \frac{C \sigma^2 k(\theta)}{n} (\log n)^8 \right). 
\end{equation}
where $C>0$ is a universal constant. 

In~\cite{CGS17} the authors also established a property of the LSE that they termed `variable'
adaptation. Let $\I_{n_1} := \left\{\theta \in 
  \R^{n_1}: \theta_1 \leq \dots \le \theta_{n_1} \right\}$. Suppose $\theta^* = (\theta^*_{ij}) \in \I$  has the property that $\theta^*_{ij}$ only
depends on $i$, i.e., there exists $\theta^{**} \in \I_{n_1}$ such that $\theta^*_{ij} = \theta^{**}_i$ for every $i$ and $j$. If we knew this fact about $\theta^*$, then the most natural way of estimating it would be to
perform vector isotonic estimation based on the row-averages $\bar{y} := \left(\bar{y}_{1}, \dots,   \bar{y}_{n_1} \right)$, where
$\bar{y}_{i} := \sum_{j=1}^{n_2} y_{ij}/n_2$, resulting in an estimator $\breve{\theta}$ of $\theta^{**}$. Note that the construction of $\breve{\theta}$ requires the knowledge
that all rows of $\theta^*$ are constant. As a consequence of the adaptive risk bound \eqref{tama}, it was shown in~\cite[Theorem
2.4]{CGS17} that the matrix isotonic LSE $\hat{\theta}$ achieves the same risk bounds as $\breve{\theta}$, up to additional logarithmic factors. This is remarkable because $\hat{\theta}$ uses no special knowledge on $\theta^*$; it automatically adapts to intrinsic dimension of $\theta^*$. 

\begin{remark}[Extension to $d \ge 2$]\label{hawa}
The recent paper, Han et al.~\cite{Han17}, studied $d$-dimensional isotonic regression for general $d \geq 1$ and proved versions of inequalities \eqref{wwc} and \eqref{tama}. Specifically, it is shown there that the worst case risk of the LSE is bounded from above by $n^{-1/d} (\log n)^4$ (ignoring  multiplicative factors involving $\sigma$ and $V(\theta^*)$). Note that for $d = 2$, this matches  the rate given by \eqref{wwc}. Interestingly, it is also shown in \cite{Han17} that the LSE is minimax  rate optimal (up to the $(\log n)^4$ factor) over the class of all bounded isotonic functions. This minimax optimality of the LSE is especially impressive because the class of all bounded isotonic functions for $d \geq 3$ is quite massive in terms of metric entropy and it was suspected previously that the LSE might suffer from overfitting. \cite{Han17} also extended the adaptive risk bound \eqref{tama} to $d \geq 3$ by proving that 
\begin{equation*}
  R(\hat{\theta}, \theta^*) \leq  \inf_{\theta \in \M}  \left(\frac{\|\theta^* -
    \theta\|^2}{n} + C_d \sigma^2 \left(\frac{k(\theta)}{n} \right)^{2/d} (\log n)^8 \right).  
\end{equation*}
Note that the $k(\theta)/n$ term in \eqref{tama} is replaced by $(k(\theta)/n)^{2/d}$ in the above bound. \cite[Proposition 2]{Han17} also observed that the above bound will not hold if $(k(\theta)/n)^{2/d}$ is replaced by $k(\theta)/n$. This implies that the LSE for $d \ge 3$ also displays adaptive behavior for piecewise hyperrectangular constant functions but that the adaptation risks are not parametric. We should also mention here that \cite{Han17} also obtained results for the isotonic LSE under random design settings. 
\end{remark}

Let us reiterate that the bounds \eqref{wwc} and \eqref{tama} are established under the assumption that the errors $\eps_{i, j}$ are i.i.d.~$N(0, \sigma^2)$. It is possible to generalize them to  sub-Gaussian errors (see \cite[Section 6]{Bellec15} for general results with sub-Gaussian errors).  However, it is not known if they hold under general error distributions that are not sub-Gaussian. Also risk bounds in other loss  functions (such as those in appropriate $\ell_p$-metrics) are not available. 

\subsection{Unimodal Regression}\label{umo}
In this subsection we summarize the two kinds of risk bounds known for the LSE in unimodal (decreasing and then increasing) regression, introduced in  Example~\ref{ex:UniReg}. The unimodal LSE $\hat \theta$ is defined as any projection of $Y$ onto $\U$, a finite union of the closed convex cones described in Example~\ref{ex:UniReg}.  It is known that $\hat \theta$ is piecewise constant with possible jumps only at the design points. Once the mode of the fitted LSE is known (and fixed), $\hat \theta$ is just the nonincreasing (isotonic) LSE fitted to the points to the left of the mode and nondecreasing (isotonic) LSE fitted to the points on the right of the mode.

As in isotonic regression, the unimodal LSE $\hat \theta$ exhibits adaptive behavior. In fact, the risk bounds for the unimodal LSE $\hat \theta$ are quite similar to those obtained for the isotonic LSE. The two kinds of risk bounds are given below (under the assumption that $\eps \sim N_n(0,\sigma^2 I_n)$):
\begin{equation}\label{eq:UniReg-WC}
R(\hat \theta, \theta^*) \le C \left(\frac{\sigma^2 V(\theta^*)}{n} \right)^{2/3} + C \frac{\sigma^2}{n^{2/3}}, \quad \mbox{where } \theta^* \in \U,
\end{equation}
and 
\begin{equation}\label{eq:UniReg-Adap}
R(\hat \theta, \theta^*) \le C \inf_{\theta \in \U}
  \left[\frac{1}{n} \|\theta^* - \theta\|^2 + C\frac{\sigma^2
      (k(\theta) +1)}{n} \log \frac{en}{k(\theta)+ 1} \right]
\end{equation}
where $k(\theta)$ is the number of constant blocks of $\theta$, $V(\theta^*) := \max_{i, j} |\theta_i^* -  \theta_j^*|$ is the range or variation of $\theta^*$ and $C>0$ is a universal constant.  

The worst case risk bound~\eqref{eq:UniReg-WC} is given in~\cite[Theorem 2.1]{CL15} while the adaptive risk bound~\eqref{eq:UniReg-Adap} is a consequence of~\cite[Theorem A.4]{Bellec15} (after integrating the tail probability). The proof of~\eqref{eq:UniReg-WC} (given in~\cite[Theorem 2.1]{CL15}) is based on the general theory of least squares outlined in Section~\ref{sec:Th-LSE}; also see~\cite[Theorem 2.2]{Chat14}. It shows that a unimodal regression function can also be estimated at the same rate as a monotone function. The adaptive risk bound~\eqref{eq:UniReg-Adap}, although being similar in spirit to that of the isotonic LSE, is weaker than~\eqref{eq:IsoAdapBd2} (obtained for the isotonic LSE). Note that  inequality~\eqref{eq:UniReg-Adap} is not sharp (i.e., the leading constant on the right side of~\eqref{eq:UniReg-Adap} is not 1); in fact it is not known whether a sharp oracle inequality can be constructed for $R(\hat \theta, \theta^*)$ (see~\cite{Bellec15}). The proof of the adaptive risk bound is also slightly more involved than that of Theorem~\ref{thm:AdapIsoReg}; the fact that the underlying parameter space $\U$ is non-convex complicates the analysis.

\subsection{Shape-restricted additive models}
Given observations $(x_1, y_1), \dots, (x_n, y_n)$ where $\{x_i = (x_{ij}, 1 \leq j \leq d)\}_{i=1}^n$ are $d$-dimensional design points and $y_1,
\dots, y_n$ are real-valued, the additive model (see e.g.,~\cite{HT90, MLN99}) assumes that
\begin{equation*}\label{admd}
  y_i = \mu^* + \sum_{j=1}^d f^*_j(x_{ij}) + \varepsilon_i \qt{for $i = 1,
    \dots, n$} 
\end{equation*}
where $\mu^* \in \R$ is an unknown intercept term, $f^*_1, \dots, f^*_d$ are unknown univariate functions satisfying 
\begin{equation}\label{iden}
\frac{1}{n}\sum_{i=1}^n f^*_j(x_{ij}) = 0, \quad \mbox{ for every }j = 1, \dots, d,
\end{equation} 
and $\varepsilon_1,\ldots, \varepsilon_n$ are unobserved mean-zero
errors. An assumption similar to~\eqref{iden} is necessary to ensure the identifiability of $f^*_1, \dots, f^*_d$. We focus our attention to shape-restricted additive models where it is assumed  that each $f^*_j$ obeys a known qualitative restriction such as
monotonicity or convexity which is captured by the assumption that
$f^*_j \in \F_j$ for a known class of functions $\F_j$. One of the main goals in additive modeling is to recover each individual function $f^*_j \in \F_j$ for $j = 1, \dots, d$.  

The LSEs $\hat{\mu}, \hat{f}_j$ of $\mu^*, f_j^*$, for $j= 1, \dots, d$ are defined as minimizers of the sum of squares criterion, i.e.,
\begin{equation}\label{eq:LSE-AM}
(\hat \mu, \hat f_1, \ldots, \hat f_d):= \argmin  
\sum_{i=1}^n \Big(y_i - \mu - \sum_{j=1}^d f_j(x_{ij})\Big)^2
\end{equation}
under the constraints $\mu \in \R, f_j \in \F_j, \sum_{i=1}^n f_j(x_{ij}) = 0 \text{ for } j = 1, \dots, d.$ It is natural to compare the performance of these LSEs to the corresponding {\it oracle} estimators defined in the following way. For each
$k = 1, \dots, d$, the oracle estimator 
$\hat{f}^{OR}_k$  is defined as 
\begin{equation}\label{eq:Oracle-AM}
\hat{f}^{OR}_k := \argmin_{f_k} \sum_{i=1}^n \Big(y_i - \mu^* -
      \sum_{j \neq k} f^*_j(x_{ij}) - f_k(x_{ik})
    \Big)^2, 
\end{equation}
where $f_k \in \F_k$ and satisfies $\sum_{i=1}^n f_k(x_{ik}) = 0$. In other words, $\hat{f}_k^{OR}$ assumes knowledge of $f^*_{j}$, for $j \neq k$, and $\mu^*$, and performs least squares minimization only over $f_k
\in \F_k$. 
%One can similarly define
%\begin{equation*}
%  \hat{\mu}^{OR} := \argmin \left\{\sum_{i=1}^n \left(y_i - \mu -
%      \sum_{j=1}^d f^*_j(x_{ij})  \right)^2 : \mu \in \R\right\} .  
%\end{equation*}

A very important aspect about shape-restricted additive models is that it is possible for the LSE $\hat{f}_k$ to be close to the oracle estimator $\hat{f}^{OR}_k$, for each $k = 1, \dots, d$. Indeed, this property was proved by Mammen and Yu \cite{mammen2007additive} under certain
assumptions for additive isotonic regression where each function $f_j$ is assumed to be monotone. Specifically,~\cite{mammen2007additive}
worked with a random design setting where the design points are assumed to be i.i.d.~from a Lipschitz density that is bounded away from zero and infinity on $[0, 1]^d$ (this is a very general setting which allows for non-product measures). They also assumed that each function $f_j$ is differentiable and strictly increasing. %and also that the errors $\eps_1, \dots, \eps_n$ are i.i.d.~having subexponential moments. 
Although the design restrictions in this result are surprisingly minimal, we believe that the assumptions on the $f_j$'s can be relaxed. In particular, this result should hold when $f_j$'s are piecewise constant and even under more general shape restrictions such as convexity. 

Our intuition is based on the following simple observation that there
exist design configurations where the LSE $\hat{f}_k$ is remarkably
close to $\hat{f}^{OR}_k$ for each $k = 1, \dots, d$ under almost no
additional assumptions. The simplest such instance is when the set of design points $\X := \{x_1, \dots, x_n\} \subseteq \R^d$ has a Cartesian product structure in the sense that $\X$ equals $\X_1 \times \dots \times \X_d$ where each $\X_i$ is a subset of the real line.  In this case, it is easy to see  
that $\hat{f}_k$ is \textit{exactly} equal to $\hat{f}^{OR}_k$ as
stated in the result below. It is convenient here to index the observations as $(i_1, \dots, i_d)$ where each $i_j$ ranges in the set $\X_j$ for $j = 1, \dots, d$. The observation model can then be written as 
\begin{equation}\label{eq:2AddMdl}
  y_{i_1, i_2, \dots, i_d} = \mu^* + f_1^*(i_1) + f_2^*(i_2) + \dots + f_d^*(i_d) +
  \varepsilon_{i_1, i_2, \dots, i_d},  
\end{equation}
for $i_n \in \X_j, j = 1, \dots, d$. The following result is proved in Section~\ref{sec:AddMdl} for the special case $d = 2$ (the proof for the general  case follows analogously).  
\begin{lemma}\label{lem:AM-Oracle}
Consider model~\eqref{eq:2AddMdl} where $f_j^* \in \F_j$, for $j=1,\dots, d$. Suppose that $\hat{f}_j, j = 1, \dots, d$ denote the LSEs of $f_j^*, j = 1, \dots, d$, as defined in~\eqref{eq:LSE-AM}. Also, let the oracle estimators $\hat{f}_j^{OR}, j= 1, \dots, d$ be as defined in~\eqref{eq:Oracle-AM}. Then $\hat{f}_j = \hat{f}_j^{OR}$ for every $j = 1, \dots, d$. 
\end{lemma}

Note that we have made no assumptions at all on $\F_1, \dots, \F_d$. Thus when the design points come from a product set $\X_1 \times \dots \times \X_d$ in $\R^d$, the LSE of $f^*_j$ is exactly equal to the oracle estimate $\hat{f}_j^{OR}$ for every $j = 1, \dots, d$. For general design configurations, it  might be much harder to relate the LSEs to the corresponding oracle estimators. Nevertheless, the aforementioned phenomenon for gridded    designs allows us to conjecture that the closeness of $\hat{f}_j$ to  $\hat{f}_j^{OR}$ must hold in much greater generality than has been observed previously in the literature. 

It may be noted that the risk behavior of $\hat{f}_j^{OR}$ is easy to characterize. For example, when $f_j^*$ is assumed to be monotone,  $\hat{f}_j^{OR}$ will satisfy risk bounds similar to those described in Section \ref{sec:IsoRisk}. Likewise, when $f_j^*$ is assumed to be convex, then $\hat{f}_j^{OR}$ will satisfy risk bounds described in Subsection \ref{cvre}. Thus, when $\hat{f}_j$ is close to $\hat{f}_j^{OR}$ (which we expect to happen under a broad set of design configurations), it is natural to expect that $\hat{f}_j$ will satisfy such risk bounds as well.

%The above lemma, in particular, implies that under the gridded design, the LSEs $\hat f_1$ and $\hat f_2$ satisfy the risk properties of the oracle estimators {\color{red} Edit this!}.

\subsection{General theory of LSEs}\label{sec:Th-LSE}
In this section, we collect some general results on the behavior of the LSEs that are useful for proving the risk bounds described in the previous two sections. These results apply to LSEs that are defined by \eqref{eq:LSE} for a closed convex constraint set $\C$. Convexity of $\C$ is crucial here (in particular, these results do not directly apply to unimodal regression where the constraint set is non-convex; see Section~\ref{umo}). We assume that the observation vector $Y = \theta^* + \eps$ for a mean-zero random vector $\eps$. Except in Lemma \ref{lem:AdapRB}, we assume that $\eps \sim N_n(0, \sigma^2 I_n)$. 

The first result reduces the problem of bounding $R(\hat{\theta}, \theta^*)$ to controlling the expected supremum of an appropriate Gaussian process. This result was proved by Chatterjee~\cite{Chat14} (see \cite{chen2017note, vanconcentration} for extensions to penalized LSEs). 

%The general theory of LSEs, subject to the constraint that $\theta \in K$, has a long history and is, by now, well established (see e.g., \cite{VandegeerBook,   vaartwellner96book, Chat14}). %This theory implies that one can bound the risk of $\hat{\theta}$ based on accurate upper bounds for   
%\begin{equation}\label{gtd}
%  G(t) := \E \left[\sup_{\theta \in \I : \|\theta - \theta^*\| \leq  t}  \langle\xi, \theta - \theta^* \rangle  \right], \quad \mbox{for $t > 0$}.
%\end{equation}
%The following result due to~\cite[Corollary 1.2]{Chat14} is a key technical tool for the proof of~\eqref{eq:IsoWC}. It reduces the problem of bounding $R(\hat{\theta},\theta^*)$ to controlling the expected supremum of an appropriate Gaussian process. %This result is easier to apply in our setting compared to older results in empirical process theory described in \citet{VandegeerBook} and \citet{vaartwellner96book}. 

\begin{lemma}[Chatterjee]\label{chatthm}
Consider the LSE \eqref{eq:LSE} for a fixed closed convex set $\C$. Assume that $Y = \theta^* + \eps$ where $\eps \sim N_n(0, \sigma^2 I_n)$ and $\theta^* \in \C$. Let us define the function $g_{\theta^*}: \R_{+} \rightarrow \R$ as 
\begin{equation}\label{eq:chat}
  g_{\theta^*}(t)  :=  \E \left[\sup_{\theta \in \C : \|\theta - \theta^*\| \leq 
  t}  \langle\varepsilon, \theta - \theta^* \rangle  \right]
  - \frac{t^2}{2}.
\end{equation}
Let $t_{\theta^*}$ be the point in $[0, \infty)$ where $t \mapsto  g_{\theta^*}(t)$ attains its maximum (existence and uniqueness of
$t_{\theta^*}$ are proved in~\cite[Theorem 1.1]{Chat14}). Then there
exists a  universal positive constant $C$ such that   
\begin{equation}\label{cim}
R(\hat{\theta}, \theta^*) \leq \frac{C}{n} \max \left(t_{\theta^*}^2, \sigma^2 \right).  
\end{equation}
\end{lemma}

\begin{remark}
  Chatterjee~\cite{Chat14} actually proved a result that is much stronger than \eqref{cim}. Specifically, he proved that the fluctuations of the random variable $\|\hat{\theta} - \theta^*\|$ around the deterministic quantity $t_{\theta^*}$ are of the order $\sqrt{t_{\theta^*}}$. When $t_{\theta^*}$ is large, this therefore implies that  $\|\hat{\theta} - \theta^*\|$ is tightly concentrated around $t_{\theta^*}$. The bound \eqref{cim} is an easy consequence of this concentration result.   
\end{remark}

Lemma \ref{chatthm} reduces the problem of bounding $R(\hat{\theta}, \theta^*)$ to that of bounding  $t_{\theta^*}$. For
this latter problem,~\cite[Proposition 1.3]{Chat14} observed that 
\begin{equation*}\label{chat2}
  t_{\theta^*} \leq t^{**} \quad \mbox{whenever $t^{**} > 0$ and $g_{\theta^*}(t^{**}) \leq 0$}.  
\end{equation*}
In order to bound $t_{\theta^*}$, one therefore seeks $t^{**} > 0$ such that $g_{\theta^*}(t^{**}) \leq 0$. This now requires a bound on the expected supremum of the Gaussian process in the definition of
$g_{\theta^*}(t)$ in~\eqref{eq:chat}. A simple upper bound for this expected Gaussian supremum is given by Dudley's entropy bound (see e.g.,~\cite[Chapter 2]{talagrand2014upper}) which is given below. This bound involves covering numbers. For a subset $K \subseteq \R^{n}$ and $\epsilon > 0$,
let $N(\epsilon, K)$ denote the $\epsilon$-covering number of $K$
under the Euclidean metric $\|\cdot\|$ (i.e., $N(\epsilon, K)$ is the
minimum number of closed balls of radius $\epsilon$ required to cover
$K$). The logarithm of $N(\epsilon, K)$ is known as the $\epsilon$-metric entropy of $K$. Also, for each  $\theta^* \in \C$  and $t > 0$, let   
\begin{equation*}\label{bn}
  B(\theta^*, t) := \left\{\theta \in \C: \|\theta - \theta^*\| \leq t \right\}
\end{equation*}
denote the ball of radius $t$ around $\theta^*$. Observe that the
supremum in the definition in~\eqref{eq:chat} is over all $\theta \in B(\theta^*, t)$. Dudley's entropy bound leads to the following upper bound for the expected Gaussian supremum appearing in the definition of $g_{\theta^*}(t)$.
\begin{lemma}[Chaining]\label{dudthm}
For every $\theta^* \in \C$ and $t > 0$,
{\begin{equation*}\label{gaup}
\E \left[ \sup_{\theta \in B(\theta^*, t)} \left<\varepsilon, \theta - \theta^* \right> \right] \leq  \sigma \inf_{0 < \delta
      \leq 2t} \left\{12 \int_{\delta}^{2t}
      \sqrt{\log N(\epsilon, B(\theta^*, t))} \;
     d\epsilon + 4 \delta \sqrt{n} \right\} .  
\end{equation*}}
\end{lemma}

\begin{remark}
  Dudley's entropy bound is not always sharp. More sophisticated \textit{generic chaining} arguments exist which gives tight bounds (up to universal multiplicative constants) for suprema of Gaussian processes; see \cite{talagrand2014upper}. 
\end{remark}

Lemma \ref{chatthm} and Lemma \ref{dudthm} present one way of bounding $R(\hat{\theta}, \theta^*)$. This involves controlling the metric entropy of subsets of the constraint set $\C$ of the form $B(\theta^*, t)$. This method is useful but works only for the case of Gaussian/sub-Gaussian errors.  

%\begin{remark}
%It may be noted that the left hand side of \eqref{gaup} is always
%less than or equal to $\sigma t$ which can be seen by an application
%of the Cauchy-Schwarz inequality. As a result, one does not need to
%consider $\delta > t/2$ in \eqref{gaup}. 
%\end{remark}

%The general results outlined here essentially reduce the problem of bounding $R(\hat{\theta},\theta^*)$ to controlling the metric entropy of subsets of $K$ of the form $B(\theta^*, t)$. 

%Such a metric entropy bound is proved in the next subsection. This is the key technical component in the proof of Theorem~\ref{kp1}. 

%{\color{red} Please polish and streamline this subsection.}
%the performance of $\hat \theta$ is different for different $\theta^*$'s. When $f$ is ``strictly'' increasing with a derivative (see e.g.,~\cite[Theorem 4.2]{CGS15}) and when $\theta^*$ is the constant sequence. 

Let us now present another result which is useful for proving adaptive risk bounds under misspecification. We shall now work with  general error distributions for $\eps$ that are not necessarily Gaussian (we only assume that $\E (\eps) = 0$). This result essentially states for bounding $R(\hat{\theta}, \theta^*)$, it is possible to work with \textit{tangent cones} associated with $\C$ instead of $\C$. It is easier to deal with cones as opposed to general closed convex sets which leads to the usefulness of this result. 

For a closed convex set $\C$ and $\theta \in \C$, the \textit{tangent cone} of $\C$ at $\theta$ is defined as 
\begin{equation*}\label{eq:TangentC2}
T_{\C}(\theta) := \mbox{Closure}\{t(\eta - \theta): t \ge 0, \eta \in \C\}.
\end{equation*}
Informally, $T_{\C}(\theta)$ represents all directions in which one can move from $\theta$ and still remain in $\C$. It is helpful to note that when $\C$ is a closed convex cone (as in many applications of shape restricted regression), then the tangent cone has the following simple expression: 
\begin{equation}\label{tobel}
  T_{\C}(\theta) = \left\{c - t \theta : c \in \C, t > 0 \right\} . 
\end{equation}
In other words, we simply add the generator $-\theta$ to the cone $\C$ to obtain $T_{\C}(\theta)$.  

The following lemma relates the risk $R(\hat{\theta}, \theta^*)$ to tangent cones. 
\begin{lemma}\label{lem:AdapRB}
Let $\C$ be a closed convex set in $\R^n$. Let $\theta^* \in \R^n$ and suppose that $Y  = \theta^* + \sigma Z$ for some mean-zero random vector $Z$ with $\E \|Z\|^2 < \infty$. Then, 
\begin{equation}\label{eq:AdapRB}
\E \big[\|\hat \theta - \theta^*\|^2 \big] \le \inf_{\theta \in \C} \Big\{\|\theta^* - \theta\|^2 + {\sigma^2}\E \big[\| \Pi_{T_{\C}(\theta)}(Z)\|^2 \big]\Big\}, 
\end{equation}
where $\Pi_{T_{\C}(\theta)}(Z)$ denotes the projection of $Z$ onto the closed convex cone $T_{\C}(\theta)$. 
\end{lemma} 
Some remarks on this lemma are given below. 
\begin{remark}[Statistical dimension]
  When $Z \sim N_n(0, I_n)$ and $K$ is a closed convex cone in $\R^n$, the quantity  
\begin{equation*}\label{eq:StatDim}
\delta(K) := \E \big[\| \Pi_K(Z)\|^2 \big] = \E \big[\langle Z, \Pi_K(Z)\rangle \big] = \E \left[ \Big(\sup_{\theta \in K:\|\theta\| \le 1} \langle Z, \theta \rangle \Big)^2 \right],
\end{equation*}
has been termed the \textit{statistical dimension} of $K$ by Amelunxen et al.~\cite{LivEdge}. Therefore, when $Z \sim N_n(0, I_n)$,inequality \eqref{eq:AdapRB} bounds the risk $R(\hat{\theta}, \theta^*)$ of the LSE via the statistical dimension of the tangent cones $T_{\C}(\theta)$. 
\end{remark}

\begin{remark}[No distributional assumptions]
There are no distributional assumptions on $Z$ for~\eqref{eq:AdapRB} to
hold. In particular, the components of $Z$ can be arbitrarily
dependent and non-Gaussian (as long as $\E \|Z\|^2 < \infty$). This follows from \cite[Proposition 2.1]{Bellec15} which is a deterministic assertion. 
\end{remark}

\begin{remark}
When $\theta^* \in \C$, then one can take $\theta = \theta^*$ in the
right side of~\eqref{eq:AdapRB} to deduce that 
\begin{equation}\label{oha}
  \E \|\hat{\theta} - \theta^*\|^2 \leq \sigma^2 \E
  \|\Pi_{T_{\C}(\theta^*)}(Z)\|^2.  
\end{equation}
This inequality \eqref{oha} was first proved by \cite{OHJournal}. Bellec \cite{Bellec15} extended it to the case when $\theta^* \notin \C$ by proving Lemma~\ref{lem:AdapRB}.    
\end{remark}

\begin{remark}[Tightness]\label{belt}
A remarkable fact proved by Oymak and Hassibi~\cite{OHJournal} is that  
\begin{equation}\label{tnt}
  \lim_{\sigma \downarrow 0} \frac{1}{\sigma^2} \E \|\hat{\theta} -
    \theta^*\|^2 = \E \|\Pi_{T_{\C}(\theta^*)}(Z)\|^2 \quad \mbox{when
    $\theta^* \in \C$}. 
\end{equation}
Analogues of this inequality when $\theta^* \notin \C$ have been recently proved in~\cite{fang2017risk}. The equality in~\eqref{tnt} implies that if $r_n(\theta)$, for $\theta \in \C$,  is
any \textit{rate} term controlling the adaptive behavior of the LSE 
in the following sense: 
\begin{equation}\label{dti}
  \E \|\hat \theta - \theta^*\|^2 \leq \inf_{\theta \in \C}
  \left\{\|\theta^* - \theta\|^2 + \sigma^2 r_n(\theta) 
  \right\} \quad \mbox{for every $\theta^* \in \R^n$}
\end{equation}
then it necessarily must happen that 
\begin{equation*}
  r_n(\theta) \geq \E \|\Pi_{T_{\C}(\theta)}(Z)\|^2 \quad \mbox{for every 
    $\theta \in \C$}. 
\end{equation*}
Thus it suffices to work with
tangent cones (i.e., focussing on bounding $\E
\|\Pi_{T_{\C}(\theta)}(Z)\|^2$) for proving adaptive risk bounds of the form \eqref{dti}. It must be noted  here though that \eqref{dti} can be quite suboptimal when $\sigma$ is large.    
\end{remark}

We shall show how to apply Lemma \ref{lem:AdapRB} to prove the adaptive risk bound \eqref{eq:IsoAdapBd} in Section \ref{sec:AdapIsoReg}. Lemma \ref{lem:AdapRB} is also crucially used in \cite{Bellec15} to prove the adaptive risk bound \eqref{yno} for convex regression. Lemma \ref{lem:AdapRB} also has applications beyond shape-restricted regression. It has been recently used to prove risk bounds for total variation denoising and trend filtering (see \cite{guntuboyina2017spatial}).

\section{Pointwise Asymptotic Theory}\label{sec:Inference}
Till now we have focused our attention on (global) risk properties of shape-restricted LSEs. In this section we investigate the pointwise limiting behavior of the estimators. By the pointwise behavior we mean the distribution of the LSE $\hat f_n$ at a fixed point (say $t$), properly normalized. Developing asymptotic distribution theory for the LSEs turns out to be rather non-trivial, mainly because there is no closed form simple expression for the LSEs; all the properties of the estimator have to be teased out from the general characterization~\eqref{eq:Charac}. 

The LSEs exhibit non-standard asymptotics: The limiting distributions that arise are non-normal (and the rates of convergence are slower than ${n}^{-1/2}$) and involve many nuisance parameters (that are difficult to estimate). As before, analyzing the isotonic LSE is probably the simplest, and we will work with this example in Section~\ref{sec:PwLimTh}. In Section~\ref{sec:ConfInt} we develop bootstrap and likelihood based methods for constructing (asymptotically) valid pointwise confidence intervals, for the isotonic regression function $f$, that bypass estimation of nuisance parameters. Section~\ref{sec:PwLimTh-Cvx} deals with the case when $f$ is convex --- we sketch a proof of the pointwise limiting distribution of the convex LSE. Not much is known in this area beyond $d=1$ for any of the shape-restricted LSEs discussed in the Introduction. 
% when it comes to finding their (asymptotic) distributions. 

\subsection{Pointwise limit theory of the LSE in isotonic regression}\label{sec:PwLimTh}
Let us recall the setup in~\eqref{eq:RegMdl} where $f$ is now an unknown nondecreasing function. Further, for simplicity, let $x_i = i/n$, for $ i =1,\ldots, n$, be the ordered design points and we assume that $\varepsilon_1,\ldots, \varepsilon_n$ are i.i.d.~mean zero errors with finite variance $\sigma^2 >0$.  The above assumptions can be relaxed substantially, e.g., we can allow for dependent, heteroscedastic errors and the $x_i$'s can be any sequence whose empirical distribution converges to a probability measure on $[0,1]$; see e.g.,~\cite{AH06},~\cite[Section 3.2.15]{VW96}.

We start with another useful characterization of the isotonic LSE (\cite[Theorem 1.1]{BBBB72}). Define the cumulative sum diagram (CSD) as the continuous piecewise affine function $F_n: [0,1] \to \R$ (with possible knots only at $i/n$, for $i = 1,\ldots, n$) for which 
\begin{equation}\label{eq:F_n}
F_n(0) := 0, \quad \mbox{and} \quad F_n\Big(\frac{i}{n}\Big) := \frac{1}{n} \sum_{j=1}^i y_j, \quad \mbox{for} \; i = 1,\ldots, n.
\end{equation} 
For any function $g: I \to \R$, where $I \subset \R$ is an interval, we denote by $\tilde g$ the {\it greatest convex minorant} (GCM) of $g$ (on $I$), i.e., $\tilde g$ is the largest convex function sitting below $g$.  Thus, $\tilde F_n$ denotes the GCM of $F_n$ (on the interval $[0,1]$). Let $\hat f_n:(0,1] \to \R$ be defined as the {\it left-hand derivative} of the GCM of the CSD; i.e., $$\hat f_n := [\tilde F_n]' \equiv \tilde F_n' ,$$ the left-hand slope of $\tilde F_n$. Then, it can be shown that (see e.g.,~\cite[Chapter 1]{RWD88}) the isotonic LSE $\hat \theta$ is given by $\hat \theta_i = \hat f_n(i/n)$, for $i = 1,\ldots, n$. Figure~\ref{fig:IsoReg-GCM} illustrates these concepts from a simple simulation. 

\begin{figure}
\includegraphics[height = 2.0in, width = 5.2in]{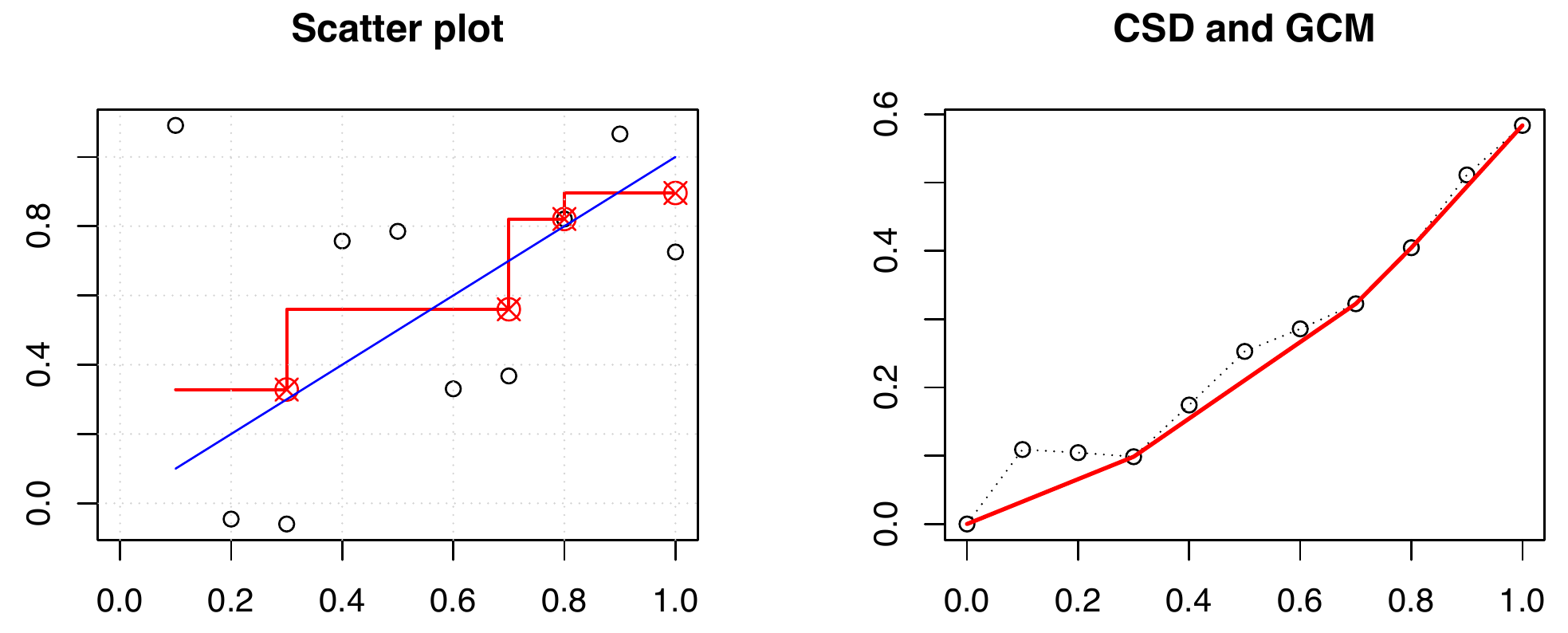}
\caption{The left panel shows the scatter plot with the fitted function $\hat f_n$ (in red) and the true $f$ (in blue) while the right panel shows the CSD (dashed) along with its GCM (in red). Here $n = 10$, $f(x) = x$ and $\eps \sim N_n(0,\sigma^2 I_n)$ with $\sigma = 0.5$.} 
\label{fig:IsoReg-GCM}
\end{figure}

Fix $0<t<1$ and suppose that $f$ has a positive continuous derivative $f'$ on some neighborhood of $t$. The following gives the asymptotic distribution of $\hat f_n(t)$, properly normalized:\begin{equation}\label{eq:LimDist}
\Delta_n := n^{1/3}\{\hat f_n(t) - f(t)\} \stackrel{d}{\to} \kappa \mathbb{C},
\end{equation}
where $\mathbb{C} := \arg \min_{h \in \R} \{\W(h) + h^2\}$ has Chernoff's distribution (here $\W(\cdot)$ is a two-sided Brownian motion starting from 0) and $\kappa := [4\sigma^2 f'(t)]^{1/3}$; see e.g.,~\cite{Brunk70, Wright81, G83, G85}. In Section~\ref{sec:Asym-IsoReg} we give an outline of a proof of~\eqref{eq:LimDist}. The first result of this type was derived in~\cite{P-Rao69} for the Grenander estimator --- the maximum likelihood estimator of a nonincreasing density in $[0,\infty)$ (see~\cite{G56}). Note that the Chernoff's random variable $\mathbb{C}$ is pivotal and its quantiles are known; see e.g.,~\cite{C64, GW01}.

\subsubsection{Other asymptotic regimes.} Observe that the assumption $f'(t) \ne 0$ is crucial in deriving the limiting distribution in~\eqref{eq:LimDist}. One may ask, what if $f'(t) = 0$? Or even simply, what if $f$ is a constant function on [0,1]? In the latter case, we can easily show that, for $t \in (0,1)$,
\begin{eqnarray*}
\sqrt{n} \{\hat f_n(t) - f(t)\} \stackrel{d}{\to} \sigma[\tilde \B]'(t),
\end{eqnarray*} 
where $\B$ is the standard Brownian motion on $[0,1]$. The above holds because of the following observations. First note that  $\sqrt{n} \{\hat f_n(t) - f(t)\}$ is the left-hand slope of the GCM of $ \sqrt{n} (F_n - F)$ at $t$ (as $F$ is now linear). As $ \sqrt{n} (F_n - F)$ converges in distribution to the process $\sigma \B$ on $D[0,1]$, we have $$\sqrt{n} \{\hat f_n(t) - f(t)\} = \sqrt{n} [\widetilde{ F_n- F}]' (t) \stackrel{d}{\to} \sigma [\tilde \B]'(t).$$ The above heuristic can be justified rigorously; see~e.g.,~\cite[Section 3.2]{GJ14}. In the related (nonincreasing) density estimation problem,~\cite{G85, CD99} showed that if  $f(t)$ lies on a flat stretch of the underlying function $f$ then the LSE (which is also the nonparametric maximum likelihood estimator, usually known as the Grenander estimator) converges to a non-degenerate limit at rate $n^{-1/2}$, and they characterized the limiting distribution. 

If one assumes that $f^{(j)}(t) =0$, for $j=1,\ldots, p-1$, and $f^{(p)}(t) \ne 0$ (for $p \ge 1$), where $f^{(j)}$ denotes the $j$'th derivative of $f$, then one can derive the limiting distribution of $\hat f_n(t)$, which now converges at the rate $n^{-p/(2p +1)}$; see e.g.,~\cite{Wright81, L82}. Note that all the above scenarios illustrate that the rate of convergence of the isotonic LSE $\hat f_n(t)$ crucially depends on the the behavior of $f$ around $t$; this demonstrates the adaptive behavior of the isotonic LSE from a pointwise asymptotics standpoint.

\subsection{Constructing asymptotically valid pointwise confidence intervals}\label{sec:ConfInt}
Although~\eqref{eq:LimDist} gives the asymptotic distribution of the isotonic LSE at the point $t$, it is not immediately clear how it can be used to construct a confidence interval for $f(t)$ --- the limiting distribution involves the nuisance parameter $f'(t)$ that needs to be estimated. A naive approach would suggest plugging in an estimator of $f'(t)$ in the limiting distribution in~\eqref{eq:LimDist} to construct an approximate confidence interval. However, as $\hat f_n$ is a piecewise constant function, $\hat f_n'$ is either 0 or undefined and cannot be used to estimate $f'(t)$ consistently. This motivates the use of bootstrap and likelihood ratio based methods to construct confidence intervals for $f(t)$. In the following we just assume that $\eps_1,\ldots, \eps_n$ i.i.d.~mean zero errors with finite variance. 

\subsubsection{Bootstrap based inference.}

Let us revisit~\eqref{eq:LimDist} and consider the problem of bootstrapping $\hat f_n$ to estimate the distribution of $\Delta_n \sim H_n$ (say). Suppose that $\hat H_n$ is an approximation of $H_n$ (which will be obtained from bootstrap in this subsection) that can be computed. Then, an approximate $1- \alpha$ ($0 <\alpha <1$) confidence interval for $f(t)$ would be $$[\hat f_n(t) - \hat q_{1-\alpha/2}n^{-1/3}, \hat f_n(t) - \hat q_{\alpha/2}n^{-1/3}],$$ where $\hat q_\alpha$ denotes the $\alpha$'th quantile of $\hat H_n$.   

In a regression setup there are two main bootstrapping techniques: `bootstrapping pairs' and `bootstrapping residuals'. Bootstrapping pairs refers to drawing with replacement samples from the data $\{(x_i,y_i): i=1, \ldots, n\}$; it is more natural when we have i.i.d.~bivariate data from a joint distribution. The residual bootstrap procedure fixes the design points $x_i$'s and draws $$y_i^* := \check f_n(x_i) + \varepsilon^*_i, \quad \; \; i = 1,\ldots, n$$ (the $^*$ indicates a data point in the bootstrap sample) where $\check f_n$ is a natural estimator of $f$ in the model, and $\varepsilon^*_i$'s are i.i.d.~(conditional on the data)~having the distribution of the (centered) residuals $\{y_i - \check f_n(x_i): i=1, \ldots, n\}$. Let $\hat f_n^*$ denote the isotonic LSE computed from the bootstrap sample. The bootstrap counterpart of $\Delta_n$ (cf.~\eqref{eq:LimDist}) is
\begin{equation*}\label{eq:Boots-Delta}
\Delta^*_n := n^{1/3} \{\hat f_n^*(t) -  \check f_{n}(t)\}.
\end{equation*}
We now approximate $H_n$ by $\hat H_n$, the conditional distribution of $\Delta^*_n$, given the data. Note that a natural candidate for $\check f_n$ in  isotonic regression is the LSE $\hat f_n$. 

Will this bootstrap approximation (by $\hat H_n$) work? This brings us to the notion of {\it consistency} of the bootstrap. Let $d$ denote the Levy metric or any other metric metrizing weak convergence of distributions. We say that $\hat H_{n}$ is {\it weakly consistent} if $d(H_n, \hat H_n)\stackrel{}{\rightarrow} 0$ in probability. If the convergence holds with probability 1, then we say that the bootstrap is {\it strongly consistent}. If $H_{n}$ has a weak limit $H$, then consistency requires $\hat H_{n}$ to converge weakly to $H$, in probability; and if $H$ is continuous, consistency requires 
$\sup_{x \in \mathbb{R}} |\hat H_{n}(x) - H(x)| \stackrel{}{\rightarrow} 0$ in probability.

It is well-known that both the above bootstrap schemes --- bootstrapping pairs and bootstrapping residuals with $\check f_n = \hat f_n$ --- yield {\it inconsistent} estimators of $H_n$; see~\cite{AH05, SBW10, Kosorok08, SX15, GH16}. Intuitively, the inconsistency of the residual bootstrap procedure can be attributed to the lack of smoothness of $\hat f_n$. Indeed a version of the residual bootstrap where one considers $\check f_n$ as a smoothed version of $\hat f_n$ (that can approximate the nuisance parameter $f'(t)$ consistently) can be shown to be consistent; see e.g.,~\cite{SX15}. Specifically, suppose that $\check f_n$ is a sequence of estimators such that
 \begin{equation}\label{eq:F_n-F}
\lim_{n\rightarrow\infty}\sup_{x \in I} \big|\check f_n(x)-f(x) \big|=0,
\end{equation} 
almost surely, where $I \subset [0,1]$ is an open neighborhood of $t$, and 
\begin{equation}\label{condf}
\lim_{n\rightarrow\infty} \sup_{h \in K} n^{1/3} \big|\check f_n(t+n^{-1/3}h)- \check f_n(t) - f'(t)n^{-1/3}h \big|=0
\end{equation} 
almost surely for any compact set $K \subset \R$. It can be shown, using arguments similar to those in the proof of~\cite[Theorem 2.1]{SX15}, that if~\eqref{eq:F_n-F} and~\eqref{condf} hold then, conditional on the data, the bootstrap estimator $\Delta_n^*$ converges in distribution to $\kappa\mathbb{C}$, as defined in~\eqref{eq:LimDist}, almost surely. Thus, this bootstrap scheme is strongly consistent.

A natural question that arises now is: Can we construct a smooth $\check f_n$ such that~\eqref{eq:F_n-F} and~\eqref{condf} hold w.p.~1? We briefly describe such a smoothed bootstrap scheme. Let $k(\cdot)$ be a differentiable symmetric  density (kernel) with compact support (e.g., $k(x) \propto (1-x^2)^2{\mathbf 1}_{[-1, 1]}(x)$) and let $K(x) := \int_{-\infty}^x k(s) \, ds$ be the corresponding distribution function. Let $h$ be a smoothing parameter. Note that $h$ may depend on the sample size $n$ but, for notational convenience, we write $h$ instead of $h_n$. Let  
$k_{h}(x) :=  k(x/h)/h  \mbox{ and } K_{h} (x) := K(x/h).$
Then the smoothed isotonic LSE of $f$ is defined as (cf.~\cite{GJW10})
\begin{equation*}\label{eq:SmoothLSE}
\check f_n (x) \equiv \check f_{n,h}(x) := \int K_{h}(x-s) \, d \hat f_n(s), \quad x \in [0,1].
\end{equation*}
It can be easily seen that $\check f_{n}$ is a nondecreasing function (if $t_2>t_1$, then $ K_{h}(t_2-s)\geq K_{h}(t_1-s)$ for all $s$).
%Figure \ref{Figurefit} gives the NPMLE and SMLE for a sample with size 1000, where we take 
 %and $h=0.3$. 
Observe that $\check f_{n}$ is a smoothed version of the step function  $\hat f_n$. 
In~\cite{SX15} it is shown that the obtained bootstrap procedure is strongly consistent, i.e., $\Delta^*_n = n^{1/3} \{\hat f_n^*(t) -  \check f_{n}(t)\}$ converges weakly to $\kappa\mathbb{C}$, conditional on the data, almost surely.

It is natural to conjecture that a (suitably) smoothed bootstrap procedure would also yield (asymptotically) valid pointwise confidence intervals for other shape-restricted regression functions (e.g., convex regression). Moreover, it can be expected that the naive `with replacement' bootstrap and the residual bootstrap using the LSE would lead to inconsistent procedures. However, as far as we are aware, there is no work that rigorously proves these claims.
%Similarly, for random elements $(V_n)_{n=1}^\infty$ and $V$ taking values in a metric space $(\mathfrak{X},\rho)$ we say that $V_n$ converges conditionally (given the data) in probability to $V$, almost surely (in probability), if for any given $\epsilon>0$, $ {P}(\rho(V_n,V)>\epsilon \mid {\mathbf Z}_n)\xrightarrow[]{} 0$ almost surely (in probability). 

\subsubsection{Likelihood ratio based inference.}
Banerjee and Wellner \cite{BW01} proposed a novel method for constructing pointwise confidence intervals for a monotone function (e.g., $f$) that avoids the need to estimate nuisance parameters; also see~\cite{Ban07, GJ15}. Specifically, the strategy is to consider the testing problem $H_0: f(t) = \phi_0$ versus $H_1: f(t) \ne \phi_0$, where $\phi_0 \in \R$ is a known constant, using the likelihood ratio statistic (LRS), constructed under the assumption of i.i.d.~Gaussian errors. If one could find the limiting distribution of the LRS under the null hypothesis and show that the limit is pivotal (as is the case in parametric models where the limiting distribution turns out to be $\chi^2$) then that would provide a convenient way to construct a confidence interval for $f(t)$ via the method of inversion: an asymptotic level $1 -\alpha$ confidence set would be given by the set of all $\phi_0$'s for which the null hypothesis $H_0: f(t) = \phi_0$ is accepted.

To study the form of the LRS, we first need to understand the {\it constrained} isotonic LSE. Consider the setup introduced in the beginning of Subsection~\ref{sec:PwLimTh} and suppose that $l := \lfloor n t \rfloor$, so that $l/n \le t < (l+1)/n$. Under $H_0:f(t) = \phi_0$, the constrained isotonic LSE $\hat f_n^0$ is given by 
\begin{equation*}\label{eq:Cons-IsoLSE}
\big\{\hat f_n^0(i/n): i = 1,\ldots, n \big\} := \argmin_{\theta \in \R^n: \theta_1 \le \cdots \le \theta_l \le \phi_0 \le \theta_{l+1} \le \cdots \le \theta_n} \sum_{i=1}^n (Y_i - \theta_i)^2.
\end{equation*}
Note that both functions $\hat f_n$ and $\hat f_n^0$ are identified only at the design points. By convention, we extend them as left-continuous piecewise constant functions
defined on the entire interval (0, 1]. The hypothesis test is based on the following LRS: 
\begin{equation*}\label{eq:LRS}
L_n := \sum_{i=1}^n \big(Y_i - \hat f_n^0(i/n)\big)^2 -  \sum_{i=1}^n \big(Y_i - \hat f_n(i/n)\big)^2.
\end{equation*}
As shown in~\cite{Ban07,Ban09} (in the setting of random design, which can be easily generalized to cover the uniform grid design; see~\cite{BBS16}), if $f(t) = \phi_0$ and $f'(t) \ne 0$, then $$L_n \stackrel{d}{\to} \sigma^2 L,$$ where $L$ is a nonnegative random variable expressible as a functional of two-sided Brownian motion plus quadratic drift $\{\W(h) + h^2: h \in \R\}$, and $\sigma^2$ is the common variance of the errors. An important feature of this limiting distribution is that it is pivotal --- free of the parameters of the problem. This readily yields confidence sets for $f(t)$ (obtained by the method of inversion) that do not need estimation of the nuisance parameter $f'(t)$ --- a challenging quantity to estimate in practice. However, an estimate of $\sigma^2$ is required, which can be easily obtained: The natural estimator $\|Y - \hat \theta\|^2/n$ of $\sigma^2$ is asymptotically normal with mean $\sigma^2$ and variance $2\sigma^4/n$ (see~\cite[Proposition 3]{MW00}). This methodology has been applied successfully in several monotone function estimation problems; see e.g.,~\cite{BW01, Ban07, GJ15}. The method, and extensions thereof,  also applies to both short- and long-range dependence regimes for the errors; see~\cite{BBS16}. Also see~\cite{BW05, BBS16} for illustrations and examples of the superior performance of the LR based method over plug-in methods for constructing confidence intervals for $f(t)$, especially when the estimation of the derivative $f'(t)$ is difficult. 

Not much is known about the (asymptotic) distribution of the LRS beyond monotone function estimation problems. However, in the recent papers~\cite{DW16,DW16b} the authors study the LRS  for testing the location of the mode of a log-concave density $f$ using the unconstrained/constrained maximum likelihood estimator of a log-concave density (also see~\cite{BRW09}) and show that, under the null hypothesis which fixes the value of the mode of $f$ (and assumes strict curvature of $- \log f$ at the mode), the LRS is asymptotically pivotal.

\subsection{Pointwise limit theory of the LSE in convex regression}\label{sec:PwLimTh-Cvx}
We assume that we have data from~\eqref{eq:RegMdl} where $f:[0,1] \to \R$ is now assumed to be convex (see Example~\ref{ex:ConvexReg}). In this section we study the pointwise asymptotic theory for the convex LSE. For simplicity, as before, we consider equi-spaced design points. Let us first describe the characterization of the convex LSE that will drive the asymptotic analysis. Given the convex LSE $\hat \theta$, let $\hat \Theta = (\hat \Theta_1,\ldots,\hat \Theta_n)$ denote the vector of its cumulative sums (divided by $n$), i.e., $\hat \Theta_i := n^{-1} \sum_{j=1}^i \hat \theta_j$, for $i = 1,\ldots, n,$ and recall $F_n$, as defined in~\eqref{eq:F_n}. Then, $\hat \theta$ has the following characterization: $\hat \theta$ is the unique vector such that $\hat \Theta_n = F_n(1)$ and 
\begin{equation}\label{eq:CvxLSE-Ch}
\sum_{i=1}^{j-1} \hat \Theta_i \begin{cases}
\ge \sum_{i=1}^{j-1} F_n(i/n) & \mbox{for }\; j = 2,\ldots, n,\\
 = \sum_{i=1}^{j-1} F_n(i/n) &	 \mbox{if $\hat \theta$ has a kink at $j/n$ or $j = n$;}
\end{cases}
\end{equation}
this follows from the characterization of projection on the closed convex set $\K$ (as defined in~\eqref{eq:CvxReg}); see~\cite[Lemma 2.6]{GJW01-a} for a complete proof. We define the convex LSE $\hat f_n:[1/n,1] \to \R$  of $f$ as the piecewise linear interpolation of the points $\{(i/n, \hat \theta_i): i = 1,\ldots, n\}$.

%Note that the vectors $\pm (1,1,\ldots, 1), \pm (1,2,\ldots, n) \in \R^n$, and $\Big\{(\overbrace{0,\ldots, 0}^\text{$i$ times},1,2,\ldots, n-i) \in \R^n: i = 1,\ldots, n-1 \Big \}$ generate the cone $\K$ (see~\eqref{eq:CvxReg}). Taking $\theta = \hat \theta \pm \alpha \pm (1,1,\ldots, 1) \in \K$ we obtain that $\sum_{i=1}^n (Y_i - \hat \theta_i) = 0$, i.e., $\hat \theta_n = n F_n(1).$ 

% and i.i.d.~mean zero sub-Gaussian errors. 

Fix $t \in (0,1)$ and consider the estimation of $f(t)$ using the convex LSE $\hat f_n(t)$ under the assumption that $f''$ is continuous and nonzero in a neighborhood of $t$. If the errors are i.i.d.~sub-Gaussian with mean zero and $f''(t) \ne 0$, the rate of convergence of $\hat f_n(t)$ is known to be $n^{-2/5}$ (see~\cite{M91}). The pointwise asymptotic distribution of the convex LSE (properly normalized) is derived in~\cite{GJW01-a}. In Groeneboom et al.~\cite{GJW01-a} the authors show that 
\begin{equation}\label{eq:CvxLim}
\Delta_n := n^{2/5}\{\hat f_n(t) - f(t)\} \stackrel{d}{\to} \h''(0),
\end{equation} 
where $\h$ is the ``invelope" of integrated Brownian motion with quartic drift ($+h^4$), and $\h''(0)$ is the second derivative of $\h$ at 0 (which exists w.p.~1). The invelope is a cubic spline lying above and touching integrated Brownian motion $+ h^4$; compare this with the ``envelope'' of Brownian  motion with a parabolic drift ($+h^2$) that appears when analyzing the isotonic LSE (see Section~\ref{sec:Asym-IsoReg}). Although a rigorous proof of the above weak convergence is long and delicate (see~\cite[Theorem 6.3]{GJW01-a}), the main intuition for such a limit can be gotten from looking at the characterization given in~\eqref{eq:CvxLSE-Ch}. We describe some of the main ideas below. The first step is to show that the characterization in~\eqref{eq:CvxLSE-Ch} can be `localized' in an appropriate sense. Then we show that the right side of the inequality in~\eqref{eq:CvxLSE-Ch}, appropriately localized and normalized, converges to a limiting process involving integrated Brownian motion $+ h^4$. Then, a continuous mapping-like result, where we look at the limiting version of the localized~\eqref{eq:CvxLSE-Ch}, yields the convergence of $\Delta_n$. A slightly more detailed sketch of the main steps is provided in Section~\ref{sec:Cvx-Reg-Asymp}. 

\begin{remark}[Multiscale inference in shape-restricted problems] The pointwise asymptotic theory for isotonic and convex regression is developed under suitable smoothness assumptions on $f$, e.g.,~\eqref{eq:LimDist} needs $f'(t) \ne 0$ whereas the weak convergence of~\eqref{eq:CvxLim} assumes $f''(t) \ne 0$. In~\cite{D03}, utilizing suitable multiscale tests, the author constructs confidence bands for $f$ that are locally adaptive in a certain sense (to the underlying smoothness in $f$) and have guaranteed coverage, assuming that $f$ is isotonic or convex. These confidence bands are computationally feasible and are also shown to be asymptotically sharp optimal in an appropriate sense. Also see the recent paper~\cite{YB17} for another method of constructing  finite-sample locally adaptive confidence bands in isotonic regression.
\end{remark}

\section{Computation of the LSE}\label{sec:Compute}
%In this section we discuss the computation of the LSE in the various shape-restricted regression problems mentioned in the Introduction. $\vspace{0.05in}$

%and (iii) briefly outline some of the other major problems/developments in this area of shape-restricted regression not addressed before in this paper.

In this section we discuss the computation of the LSE $\hat \theta$ in nonparametric shape-restricted regression problems. Note that in most cases (see e.g., Examples~\ref{ex:IsoMdl}--\ref{ex:AddMdl}) the LSE $\hat \theta$ is the projection of $Y$ onto $\C$, a (finite union of) closed convex set(s) in $\R^n$. If $\C$ is a polyhedral convex set, then the computation of $\hat \theta$ involves solving a {\it quadratic program} with a bunch of {\it linear constraints}. Many off-the-shelf solvers --- e.g., CPLEX, MOSEK, Gurobi --- can solve these quadratic programs easily even for moderately large sample sizes (e.g., $n \approx 10^5$). In the following we consider the main examples in the Introduction and discuss some problem specific algorithms that are computationally more efficient. $\vspace{0.05in}$

{\bf Isotonic and unimodal regression.} For the monotone regression problem~\cite{BBBB72}  presented a graphical interpretation of the isotonic LSE (defined in~\eqref{eq:IsoCharac}) in terms of the GCM of the CSD; see Section~\ref{sec:PwLimTh}. The method of successive approximation to the GCM can be described algebraically as the pool-adjacent-violators algorithm (PAVA); see e.g.,~\cite[p.~9-10]{RWD88}. Roughly speaking, PAVA works as follows. We start with $y_1$ on the left. We move to the right until we encounter the first violation $y_i > y_{i+1}$. Then we
replace this pair by their average, and back-average to the left as needed, to get monotonicity. We continue this process to the right, until finally we reach $y_n$. If skillfully implemented, PAVA has a
computational complexity of $O(n)$; see~\cite{Stout14} for a comparison of various algorithms to solve isotonic regression in $\ell_p$-metrics, for $p \ge 1$. The \texttt{isoreg} command in the {\it stats} package in the R programming language implements the isotonic LSE. Further, see~\cite{Stout08} for an efficient (requiring only $O(n)$ time) computation of the unimodal LSE (Example~\ref{ex:UniReg}). $\vspace{0.05in}$

{\bf Order preserving regression on a partially ordered set.} Given a partial order $\precsim$ on the design points $x_i$'s we can compute the LSE $\hat \theta$ of the isotonic (order preserving) $\theta^*$ by solving~\eqref{eq:LSE}. Here $\C$, the space where $Y$ is projected onto to obtain $\hat \theta$, is a closed convex cone and can be represented as $$\C := \{(\theta_1,\ldots, \theta_n) \in \R^d: \theta_i \le \theta_j \mbox{ if } x_i \precsim x_j, \mbox{ for some } i \ne j\}.$$ Thus, the computation of $\hat \theta$ involves solving a quadratic program with $O(n^2)$ linear constraints (although for some special situations, like isotonic regression in $d=1$, $\C$ can be represented by $O(n)$ linear constraints). The computation of the order preserving LSE on a partially ordered set (Example~\ref{ex:RegPartialOrder}) has received quite a bit of attention recently; see e.g.,~\cite{KRS15, Stout15} and the references therein. In particular,~\texttt{https://github.com/sachdevasushant/Isotonic} gives an implementation of the isotonic LSE using interior point methods. The special case of the matrix isotonic LSE defined in~\eqref{eq:MatLSE} can be computed efficiently by an iterative algorithm (see e.g.,~\cite{G70} and~\cite[Chapter 1]{RWD88}).

Once we obtain $\hat \theta = (\hat \theta_1,\ldots, \hat \theta_n)$ we can then easily construct an estimate $\hat f_n$ of the order preserving $f$ at any $x$ (not necessarily a design point) by taking a maximum over a selected number of coordinates of $\hat \theta$: We can define $\hat{f}_n$ as 
\begin{equation*}\label{eq:DefPhi}
\hat{f}_n(x)  := \sup_{j: x_j \precsim \; x} \hat \theta_j, \quad \mbox{for } x \in \R^d,
\end{equation*}
where we take the convention that $\sup(\emptyset) = -\infty$. Note that $\hat{f}_n$ is indeed order preserving --- for $u, v$ with $u \precsim v$ we have $ \hat{f}_n(u)  = \sup_{j: x_j \precsim \; u} \hat \theta_j \le \sup_{j: x_j \precsim \; v}\hat \theta_j = \hat{f}_n(v), $
as in the right side the supremum is taken over a bigger set. $\vspace{0.05in}$
%Thus, once $\hat \theta$ is computed, evaluating the function at any fixed point ( just involves .

%{\color{red} Convex polyhedron; can be reduced to solving a quadratic program with linear constraints; define the matrix $A$ }

{\bf Convex regression.}
Algorithms for the computation of the convex LSE (Example~\ref{ex:ConvexReg}) when $d=1$ can be found in~\cite{Dykstra83},~\cite{FraserM89},~\cite{GJW08} and the references therein. When $d>1$, the problem is substantially harder: Due to the lack of a natural ordering of points in $\R^d$ (for $d >1$), the constraint set $\C$ is not easy to express (cf.~\eqref{eq:CvxReg}). In fact, in this case $\C$ can be expressed as the {\it projection} of the higher-dimensional polyhedron
\begin{equation*}\label{eq:Q_cvxreg}
\left\{(\xi,\theta) \in \R^{dn +n} :\xi=[\xi_1^\top, \ldots, \xi_n^\top]^\top, \theta_j + \xi_j^\top (x_i-x_j) \leq \theta_i,\forall \;i,j=1,\dots,n \right\}
\end{equation*}
onto the space of $\theta$; see~\cite{CLS15}. The above characterization can be seen as a consequence of the subgradient inequality for convex functions; see~\cite[Theorem 25.1, p.~242]{Rock70}. Thus the computation of the convex LSE $\hat \theta$ involves  solving a quadratic program with $n(d+1)$ variables and $n(n-1)$ linear constraints; see~\cite{SS11} for the characterization, computation and consistency of the convex LSE where off-the-shelf interior point solvers (e.g., {\texttt {cvx}}, {\texttt {MOSEK}}, etc.) were used to compute $\hat \theta$. However, these off-the-shelf solvers do not scale well and become prohibitively expensive for $n \ge 300$ mainly due to the presence of $O(n^2)$ linear constraints. In~\cite{MCIS15}, exploiting problem specific structure, the authors propose a scalable algorithmic framework based on the augmented Lagrangian method to compute the convex LSE $\hat \theta$. This iterative algorithm can compute the LSE with $n \sim 5000$ and $d \sim 10$ within moderate accuracy (i.e., 4 significant digits) in around 30 minutes in a laptop. $\vspace{0.05in}$ %{\color{red}\textbf{ It will good to mention the sizes of $n$ and $p$ that the algorithm in \cite{MCIS15} can handle. }}  $\vspace{0.05in}$

{\bf Shape constrained additive models.} The computation of the additive shape-restricted (Example~\ref{ex:AddMdl}) LSE is discussed in~\cite{Meyer13b, CS16}. As this reduces to solving a quadratic program with $O(n)$ linear constraints, off-the-shelf solvers can be effectively used for computing the LSE. The shape-restricted LSE can also be computed efficiently by the back-fitting algorithm (\cite{BF85}) --- a simple iterative procedure used to fit a generalized additive model --- which involves fitting one function at a time (out of the $d$ many univariate nonparametric functions). $\vspace{0.05in}$

{\bf Shape-restricted single index model.} Let us now look at Example~\ref{ex:SIM}. Here interest focuses on estimating the nonparametric (shape-restricted) function $m$ and the finite-dimensional parameter $\beta^*$. Although single index models are well-studied in the statistical literature (see e.g.,~\cite{Powelletal89},~\cite{LiDuan89},~\cite{cuietal11} and the references therein), estimation and inference in shape-restricted single index models are not very well-developed, despite their numerous applications. %The earliest reference we could find was the work of~\cite{VANC}, where the authors considered a penalized likelihood approach in the current status regression model with a monotone link function. 
The LSE in the monotone single index model is  defined as 
\begin{equation}\label{eq:MonoSIM}
(\hat m, \hat \beta) := \argmin_{\psi,\beta} \sum_{i=1}^n (y_i - \psi(x_i^\top\beta))^2,
\end{equation} 
where the minimization is over all nondecreasing functions $\psi: \R \to \R$ and over $\beta \in \R^d$ (with $\|\beta\| = 1$, for identifiability). As the above LSE solves a non-convex problem, its computation is non-trivial. A version of the following alternating minimization scheme is typically applied to compute the LSE. For a fixed $\beta$, the sum-of-squared errors can be easily minimized over all nondecreasing functions (as this reduces to the problem to univariate isotonic regression). However, the minimization of the profiled least squares criterion (over $\beta$), for a fixed $\psi$, is non-smooth and non-convex; see~\cite{Kakade11, groeneboom2016current} and the references therein for some strategies to find $\hat \beta$. A similar strategy is employed in computing the maximum likelihood estimate in the related problem of current status regression; see e.g.,~\cite{groeneboom2016current}. Also see~\cite{KPS17} and the R package \texttt{simest} for the related computation in a convex single index model.

\section{Some open problems}\label{sec:Open}
In this section we state and motivate a few open research problems and some possible future directions. $\vspace{0.05in}$

{\bf Beyond i.i.d.~Gaussian errors.}
We have mentioned in Section \ref{sec:IsoRisk} that risk bounds for isotonic regression do not assume  that the errors $\eps_1, \dots, \eps_n$ are Gaussian. Indeed, the worst case risk bound \eqref{eq:IsoWC} as well as the adaptive risk bound in Theorem \ref{thm:AdapIsoReg} work under no distributional assumptions on the errors (it is only assumed that the errors are i.i.d.~with mean zero and finite variance $\sigma^2$; even the i.i.d.~assumption can be relaxed considerably; see \cite{Zhang02}). However the risk bounds for other shape-restricted regression problems (including convex regression, isotonic regression on partially ordered sets such as multivariate isotonic regression, unimodal regression, etc.) assume Gaussian (or sub-Gaussian) errors. Based on the results for the univariate isotonic LSE, we believe that the assumption of Gaussianity should not really be necessary for these other problems as well. However the existing proof techniques for these risk bounds strongly rely on the assumption of sub-Gaussianity. It will be very interesting to prove risk bounds in these problems without Gaussianity. We believe that new techniques will need to be developed for this. $\vspace{0.05in}$

{\bf Beyond $\ell_2$-loss.}  Most of the risk results available in the shape constrained literature apply only to the $\ell_2$-loss. A notable exception is the case of isotonic regression where risk bounds are available under the $\ell_p$-loss for every $p \geq 1$ (as already described in Subsection \ref{pp}). It will be interesting to develop risk results for $\ell_p$-losses in problems such as convex regression and multivariate shape-restricted regression. The results for isotonic regression (see Theorem \ref{thm:L_p} and the following discussion) indicate that adaptation risk bounds for LSEs have a different relationship with oracle risk bounds for $p \neq 2$. For example, for $p < 2$, the isotonic LSE is suboptimal only by a constant factor in comparison to the oracle while for $p > 2$, the isotonic LSE is significantly suboptimal. We believe that it is quite non-trivial to study risk of the LSEs under $\ell_p$-loss functions for $p \neq 2$. The existing abstract theory for studying LSEs seems to give risk results only under the $\ell_2$-loss function. $\vspace{0.05in}$

{\bf Minimax Results.}  The risk of a LSE  $R(\hat{\theta}, \theta^*)$  in a shape-restricted regression problem usually varies quite significantly as $\theta^*$ varies over the parameter space. For example, in isotonic regression, the risk behaves as $n^{-2/3}$ when $V(\theta^*)$ is bounded and as $(k/n) \log(n/k)$ when $\theta^*$ is piecewise constant having $k$ constant pieces. Minimax lower bounds over these parameter classes allow the assessment of optimality of the LSE compared to other estimators. We mentioned in Remark \ref{reim} that the isotonic LSE is minimax optimal over $\{\theta \in \I : \theta_n - \theta_1 \leq V\}$ for a wide range of values of $V$. In an interesting recent paper \cite{gao2017minimax}, the authors characterized the minimax risk over the class of all monotone vectors with at most $k$ constant pieces. Their results imply that the risk $(k/n) (\log (n/k))$ achieved by the isotonic LSE over this class is only suboptimal by a factor of $(\log n)/(\log \log n)$ in comparison to the minimax risk. Minimax lower bounds exist for other shape-restricted regression problems (see e.g., \cite{GS15, CGS15, Bellec15, Han17, CL17}) which suggest that the LSE is nearly minimax optimal but some of these results are not as tight as the corresponding results for univariate isotonic regression. It will be interesting to develop tight minimax results for other shape-restricted regression problems which will allow a precise evaluation of the minimaxity properties of the LSEs. 
% Interesting results in this direction have been obtained recently by \cite{gao2017minimax} and~\cite{CL17}. In \cite{gao2017minimax},  the authors characterize the minimax risk over the class of all monotone vectors $\theta^*$ which have at most $k$ constant pieces. Their results imply that the isotonic LSE is suboptimal for estimating such sequences by a factor of $(\log \log (n/k))/(\log (n/k))$. In \cite{CL17}, finite sample minimax lower bounds have been presented for isotonic regression for a variety of parameter classes. It will be interesting to develop similar minimax results for other shape-restricted regression problems.}
 $\vspace{0.05in}$

%Although most existing risk bounds for shape-restricted LSEs (except isotonic regression) assume Gaussian (or sub-Gaussian) errors, the bound \eqref{thm:AdapIsoReg} (in Theorem~\ref{thm:AdapIsoReg}) suggests that Gaussianity may not be crucial for obtaining such bounds. However, this In fact, for the isotonic LSE it might be possible to obtain such bounds even for certain kinds of dependent errors following the arguments of~\cite{Zhang02}. Moreover, 

%{\bf Statistical dimension of convex cones.} In bounds similar to~\eqref{eq:IsoAdapBd} for other problems (such as convex regression), the constant in front of the second term is usually not tight. $\vspace{0.05in}$

{\bf Estimation of other shape constrained regression functions.} In the recent years there has been quite a bit of interest in studying different shape-restricted regression functions, beyond $d=1$. We have already seen a few such examples in this paper (e.g., Examples~\ref{ex:RegPartialOrder},~\ref{ex:ConvexReg} and~\ref{ex:AddMdl}). What are other useful shape-restrictions in multi-dimension? In the following we mention a few such shape constraints (that have many real applications): (i) unordered weak majorization, (ii) quasiconvexity, and (iii) supermodularity.  

{\it Unordered weak majorization.} In Example~\ref{ex:RegPartialOrder} we discussed the problem of estimating an order preserving regression function (with respect to a partial order). Robertson et al.~\cite[Chapter 1]{RWD88} gives a nice overview of the properties of the LSE in this problem. As an example we introduced a generalization of monotonicity  beyond $d =1$, namely, coordinate-wise monotonicity. In the following we introduce and characterize another related (and slightly stronger) notion of monotonicity in multi-dimensions that is closely tied to the concept of majorization and Schur-convexity (see e.g.,~\cite{Ineq11}). We define the {\it unordered weak majorization} partial order $\precsim$ as
\begin{equation*}
(u_1,\ldots, u_d) \precsim (v_1,\ldots, v_d)\;\; \mbox{if and only if}\;\;\sum_{k=1}^i u_k \le \sum_{k=1}^i v_k \;\; \mbox{for } i=1,2,\ldots, d.
\end{equation*}
The following result characterizes all functions that preserve the ordering $\precsim$.
\begin{theorem}\label{thm:PO}
Let $f:\R^d \to \R$ be a continuously differentiable function. Then $f$ preserves the partial order $\precsim$ if and only if for any $z\in \R^d$, 
\begin{equation}\label{eq:DecDeriv}
f_{(1)}(z) \ge f_{(2)}(z) \ge \ldots \ge f_{(d)}(z)\ge 0,
\end{equation} 
where $f_{(i)}$ denotes the partial derivate of $f$ with respect to the $i$'th coordinate.
\end{theorem} 

The above result shows that in a regression setup if $f$ can be assumed to obey~\eqref{eq:DecDeriv}, i.e., the influence of the predictor variables is ordered, LS estimation under the unordered weak majorization partial order can be used to estimate $f$. Constraints like~\eqref{eq:DecDeriv} appear quite often in econometrics; see e.g.,~\cite{Y03} and the references therein.

{\it Quasiconvexity.} A function $f:\R^d \to \R$ is quasiconvex if and only if its sub-level sets $S_{\alpha }(f)=\{u \in \R^d : f(u)\leq \alpha \}$ are convex for every $\alpha \in \R$; see~\cite[Section 3.4]{Boyd04}. Alternatively, a function $f$ is quasiconvex if and only if $f(\alpha u+(1-\alpha)v) \le \max\{f(u),f(v)\}$, for all $u,v \in \R^n$,  and $\alpha \in [0,1]$ (cf.~\eqref{eq:CvxDef}). Quasiconvex functions extend the notion of unimodality to multi-dimensions and have applications in mathematical optimization and economics. The computation of $f$ using the method of least squares is likely to be a non-convex problem.

%A function $f:\R^d \to \R$ is said to be pseudoconvex if it is once-differentiable and $\nabla f(u)^\top(v-u) \ge 0$ implies $f(v) \ge f(u)$, for all $u,v \in \R^n$.

{\it Supermodularity.} A function $f: \R^{d} \to \R$ is supermodular if $f(u\vee v)+f(u \wedge v) \geq f(u)+f(v)$ for all $u, v \in \R^{d}$, where $u \vee v$ and $u \wedge v$ denote the component-wise maximum and minimum of $u$ and $v$, respectively, i.e., $(u_1,\ldots, u_d) \vee (v_1,\ldots, v_d) := (\max \{u_1,v_1\}, \ldots, \max\{u_d,v_d\})$ and $(u_1,\ldots, u_d) \wedge (v_1,\ldots, v_d) := (\min \{u_1,v_1\}, $ $\ldots, \min\{u_d,v_d\})$. The concept of supermodularity is used in the social sciences (economics and game theory). If $f$ is twice continuously differentiable, then supermodularity is equivalent to the condition $\frac {\partial^2f(u)}{\partial u_{i}\,\partial u_{j}} \geq 0$ for all $i \neq j$; see e.g.,~\cite{Simchi05}.

In all the above problems, computation of the LSE, its theoretical properties (consistency, rates of convergence, etc.) are unknown. $\vspace{0.05in}$

{\bf Connection to nonnegative least squares.} In many shape-restricted regression problems the LSE $\hat \theta$ is defined as the projection of $Y$ onto a closed convex polyhedral cone $\C$; see Examples~\ref{ex:IsoMdl}--\ref{ex:ConvexReg} and~\ref{ex:AddMdl}. As every closed convex polyhedral cone $\C$ in $\R^n$ can be represented in terms of its generators (i.e., there exists a finite subset of $\C$, whose elements are referred to as generators of $\C$, such that every vector $\theta \in \C$ is a nonnegative linear combination of the generators; see, for example, \cite[Corollary 7.1a]{Alex86}), $\hat \theta$ can be thought of as solving a nonnegative least squares problem. Moreover, in the examples mentioned above, $\C$ is generated by at least $O(n)$ vectors (which form the design matrix) that are highly correlated; see e.g.,~\cite{Meyer99} for the exact form of the generators of some of the examples discussed. In fact, for isotonic and convex regression in $d=1$ there are exactly $n+1$ and $n+2$ generators, respectively. However, for their higher dimensional analogues (i.e., $d>1$) it is not clear what the generators are. We think this is an open problem. More generally, one can ask how does one construct the design matrix (or the generators) corresponding to any closed convex polyhedral cone expressed in terms of linear inequalities, e.g.,  $\C := \{\theta \in \R^n: A \theta \le 0\}$ where $A$ is an $m \times n$ matrix ($m$ being the number of linear constraints) and the `$\le$' is interpreted coordinate-wise. 

As the number of generators (or the columns of the design matrix) is increasing with $n$, we are essentially solving a `high-dimensional' nonnegative least squares problem; see e.g.,~\cite{SH13}. A general open question is: Can a theory be developed on the estimation accuracy of the LSE $\hat \theta$ based solely on the properties of the design matrix? It may be noted here that the generators can be highly correlated. $\vspace{0.05in}$

{\bf Boundary behavior of shape-restricted LSEs.} It is well known that the isotonic LSE is {\it inconsistent} at the boundary of the covariate domain, i.e., $\hat f_n(0+)$ does not consistently estimate $f(0+)$ (see e.g.,~\cite{WS91, KL06, BFJPSW11} for detailed discussions on the properties of the LSE for a nonincreasing density near 0). Intuitively, this inconsistency is because there are very few  `constraints' near the boundary (of the covariates). This phenomenon is expected to persist for other shape constrained LSEs, especially in multi-dimensional problems. However not much is known about the boundary behavior of these LSEs. Even in one-dimensional convex regression, as far as we are aware, whether $\hat f_n(0+)$ is a $O_p(1)$ random variable is not known; see~\cite{GS16} for some results on $\hat f_n(0+)$ and its derivative (also see~\cite{B07}). This has motivated the study of bounded/penalized shape-restricted LSEs; see e.g.,~\cite{CLS15, WMO15, KPS17}. $\vspace{0.05in}$

{\bf Shape-restricted single index models.} 
Although several smoothing based methods have been proposed and investigated in single index models (see e.g.,~\cite{VANC, groeneboom2016current} and the references therein) to obtain $\sqrt{n}$-consistent and {\it efficient} estimators of $\beta^*$ (see~\cite{Newey90} for a brief overview of the notion of semiparametric efficiency), not much is known for just shape-restricted single index models. %However, these methods require the delicate choice of tuning parameter(s). 
Durot et al.~\cite{2016arXiv161006026B} studied the  LSE in a monotone single index model (see~\eqref{eq:MonoSIM}) and showed the $n^{1/3}$-consistency of the LSEs of $m$ and $\beta^*$; also see~\cite{VANC} and~\cite{groeneboom2016current}. However many open questions remain. The limiting distribution of the LSE $\hat \beta$ of $\beta^*$ is unknown; in fact, it is not known whether $\hat \beta$ is $\sqrt{n}$-consistent.

%In~\cite{groeneboom2016current} the authors propose a $\sqrt{n}$-consistent and asymptotically normal but \textit{inefficient} estimator of the index vector $\beta^*$ in the current status model based on the (non-smooth) maximum likelihood estimate (MLE) of the nonparametric component under just the monotonicity constraint.

In a convex single index model (i.e., $m$ is convex),~\cite{KPS17} shows that the Lipschitz constrained convex LSE (where we minimize the least squares criterion over the class of all $L$-Lipschitz convex functions, for $L$ fixed) yields a semiparametrically efficient estimator of the index parameter $\beta^*$. However, the behavior of the convex LSE (without the Lipschitz assumption) is unknown. 

As mentioned above, the computation of the shape-restricted LSEs is non-trivial. Usually an alternating minimization scheme is used to compute the LSEs. However, no convergence  guarantees (to a local optimum) exist for such an alternating minimization procedure.

\section*{Acknowledgements}
We would like to thank Moulinath Banerjee, Sabyasachi Chatterjee and Probal Chaudhuri for helpful discussions. We also thank the Associate Editor and two anonymous referees for many insightful comments and suggestions.

\appendix

\section{Proofs}\label{sec:Proofs}
This section contains proofs of some of the results from the main paper. 
\subsection{Proof of Theorem~\ref{thm:AdapIsoReg}}\label{sec:AdapIsoReg}
To prove the adaptive risk bound \eqref{eq:IsoAdapBd} for
isotonic regression we shall apply inequality \eqref{eq:AdapRB} with $Z = \eps/\sigma$ which reduces the task to
proving that  
\begin{equation}\label{gaga}
  \E \|\Pi_{T_{\I}(\theta)}(Z)\|^2 \leq 4 k(\theta) \log
  \frac{en}{k(\theta)} \quad \mbox{for every $\theta \in \I$}.
\end{equation}
To prove this, we
obviously need to characterize the tangent cone $T_{\I}(\theta)$. Fix
$\theta \in \I$ and let $k = k(\theta)$. This means that $\theta$ has
$k$ constant pieces (and $k-1$ jumps). Let the lengths of the $k$
constant pieces be $n_1, \dots, n_k$ (so that $n_j \geq 1$ and $\sum_{j=1}^k
n_j = n$). It is now an easy exercise to verify (via \eqref{tobel}) that the tangent cone 
$T_{\I}(\theta)$ is given by 
\begin{equation}\label{tt}
  T_{\I}(\theta) = \I_{n_1} \times \dots \times \I_{n_k}
\end{equation}
where $\times$ denotes Cartesian product and, for an integer $\ell
\geq 1$,  
\begin{equation*}
  \I_{\ell} := \left\{(x_1, \dots, x_{\ell}) \in \R^l : x_1 \leq \dots \leq
    x_{\ell} \right\}. 
\end{equation*}
The equality \eqref{tt} immediately implies that
\begin{equation}\label{tem1}
 \E \|\Pi_{T_{\I}(\theta)}(Z)\|^2 = \sum_{j=1}^k \E \|\Pi_{{\I_{n_j}}}(Z^{(j)})\|^2
\end{equation}
where $Z^{(j)}$ denotes the part of the vector $Z$ corresponding to
the $j^{th}$ constant piece of $\theta$. The key now is to prove that 
\begin{equation}\label{kmm}
  \E \|\Pi_{\I_n}(Z)\|^2 \leq 4 \left( 1 + \frac{1}{2} + \dots + \frac{1}{n}\right) 
  \quad \mbox{for every $n \geq 1$}. 
\end{equation}
This, along with the standard fact that $\sum_{j=1}^n (1/j) \leq \log
(en)$, implies that
\begin{equation*}
  \E \|\Pi_{{\I_{n_j}}}(Z^{(j)})\|^2 \leq 4 \log (e n_j).
\end{equation*}
Thus, we obtain, via \eqref{tem1}, 
\begin{equation*}
  \E \|\Pi_{T_{\I}(\theta)}(Z)\|^2 \leq \sum_{j=1}^k 4 \log (e n_j) \leq 4 k \log \frac{en}{k}
\end{equation*}
where the last inequality follows from the concavity of the logarithm function. This proves \eqref{gaga}
and consequently \eqref{eq:IsoAdapBd}. 
%Now let us prove \eqref{nga} under the assumption that $\epsilon_1,
%\dots, \epsilon_n$ are i.i.d with mean zero and variance $\sigma^2$
%but not necessarily Gaussian. It may be noted that the only place
%where Gaussianity of the errors was invoked in the previous argument
%was to claim that \eqref{kmm} holds with equality. Therefore, for
%proving \eqref{nga}, it is enough to establish that 
%\begin{equation}\label{gkmm}
%  \E \|\Pi_{\I_n}(Z)\|^2 \leq 4 \left(1 + \frac{1}{2} + \dots +
%    \frac{1}{n} \right)
%  \quad \mbox{for every $n \geq 1$}
%\end{equation}
%under the assumption that $Z_1, \dots, Z_n$ are i.i.d with mean zero
%and variance $\sigma^2$ (without assuming Gaussianity). For this, let  
To prove~\eqref{kmm} let 
\begin{equation*}
  U = (U_1, \dots, U_n) := \Pi_{\I_n}(Z). 
\end{equation*}
The representation \eqref{eq:IsoCharac} gives 
\begin{equation*}
  U_j = \min_{v \geq j} \max_{u \leq j} \bar{Z}_{uv} \quad \mbox{for every $1
    \leq j \leq n$}
\end{equation*}
where $\bar{Z}_{uv}$ represents the mean of $Z_u, \dots, Z_v$. It is
clear from here that 
\begin{equation*}
  U_j \leq \max_{u \leq j} \bar{Z}_{un}
\end{equation*}
so that 
\begin{equation*}
  \E (U_j)_+^2 \leq \E \max_{u \leq j} (\bar{Z}_{un})_+^2 . 
\end{equation*}
Here $x_+ := \max(x, 0)$ and $a_+^2 := (a_+)^2$. The key now is to
realize that the sequence of random variables $\bar{Z}_{1n}, \dots,
\bar{Z}_{jn}$ is a martingale (easy to verify) so that by Doob's
submartingale maximal inequality, we deduce that 
\begin{equation*}
  \E (U_j)_+^2 \leq 4 \E(\bar{Z}_{jn})_+^2. 
\end{equation*}
By an analogous argument, we can also deduce that
\begin{equation*}
  \E(U_j)_-^2 \leq 4 \E (\bar{Z}_{1j})_-^2
\end{equation*}
where $x_- := \max(0, -x)$. Putting together, we deduce 
\begin{equation*}
  \E \|\Pi_{\I_n}(Z)\|^2 = \sum_{j=1}^n \E U_j^2 \leq 4 \sum_{j=1}^n
  \left(\E (\bar{Z}_{jn})_+^2 + \E (\bar{Z}_{1j})_-^2 \right). 
\end{equation*}
Now because $Z_1, \dots, Z_n$ are i.i.d.~$(Z_1, \dots, Z_n)$ has the
same distribution as $(Z_n, \dots, Z_1)$ so that 
\begin{equation*}
  \E \|\Pi_{\I_n}(Z)\|^2 \leq 4 \sum_{j=1}^n
  \left(\E (\bar{Z}_{1,n-j+1})_+^2 + \E (\bar{Z}_{n-j+1,n})_-^2  \right) = 4 \sum_{j=1}^n
  \left(\E (\bar{Z}_{1j})_+^2 + \E (\bar{Z}_{jn})_-^2  \right). 
\end{equation*}
Adding the above two inequalities, we obtain
\begin{equation*}
  2 \E \|\Pi_{\I_n}(Z)\|^2 \leq 4 \sum_{j=1}^n \left(\E
    (\bar{Z}_{jn})^2 + \E (\bar{Z}_{1j})^2 \right) = 4 \sum_{j=1}^n
  \left(\frac{1}{n-j+1} + \frac{1}{j} \right) = 8 \sum_{j=1}^n \frac{1}{j}. 
\end{equation*}
This proves \eqref{kmm} and completes the proof of~\eqref{eq:IsoAdapBd}. \qed

\subsection{Proof of inequality~\eqref{eq:IsoWC}}\label{sec:IsoReg-WC}
The initial steps of this proof follow the same steps as in the
beginning of the proof of Theorem \ref{thm:L_p} (the proof of Theorem \ref{thm:L_p} is provided in Subsection \ref{sec:R_p-Risk}). Specifically, we use
inequality \eqref{sho} for $p = 2$ which says that 
\begin{equation*}
    \E_{\theta^*} \left(\hat{\theta}_j - \theta^*_j \right)_+^2 \leq
  \E_{\theta^*} \left(\left(\bar{\theta^*}_{j, j+m} -
    \theta^*_j \right) + \left(\max_{u \leq j} 
  \bar{\eps}_{u, j+m} \right)_+ \right)^2. 
\end{equation*}
for every $1 \leq j \leq n$ and $0 \le m \leq n-j$. The elementary
inequality $\E (c + X)^2 \le (c + \sqrt{\E X^2})^2$ for $c \geq 0$
applied to $c = \bar{\theta^*}_{j, j+m} - \theta^*_j \geq 0$ and $X := \left(\max_{u \leq j} 
  \bar{\eps}_{u, j+m} \right)_+$ now gives 
\begin{equation*}
   \E_{\theta^*} \left(\hat{\theta}_j - \theta^*_j \right)_+^2 \leq
   \left(\bar{\theta^*}_{j, j+m} - \theta^*_j + \sqrt{\E \max_{u \leq
         j} \left(\bar{\eps}_{u, j+m} \right)^2_+} \right)^2. 
\end{equation*}
Doob's maximal submartingale inequality applied to the expectation
above results in 
\begin{equation*}
  \E \max_{u \leq j} \left(\bar{\eps}_{u, j+m} \right)^2_+ \leq 4
  \E \left(\bar{\eps}_{u, j+m} \right)^2_+  \leq 4
  \left(\bar{\eps}_{u, j+m} \right)^2 = \frac{4 \sigma^2}{m + 1}. 
\end{equation*}
We have thus proved that 
\begin{equation}\label{ann1}
   \E_{\theta^*} \left(\hat{\theta}_j - \theta^*_j \right)_+^2 \leq
   \min_{0 \leq m \leq n-j} \left(\bar{\theta^*}_{j, j+m} - \theta^*_j
     +\frac{2 \sigma}{\sqrt{m+1}}\right)^2. 
\end{equation}
Similarly, we can prove that
\begin{equation}\label{ann2}
   \E_{\theta^*} \left(\hat{\theta}_j - \theta^*_j \right)_-^2 \leq
   \min_{0 \leq m \leq j-1} \left(\theta^*_j - \bar{\theta^*}_{j-m, j}
     +\frac{2 \sigma}{\sqrt{m+1}}\right)^2. 
\end{equation}
For each $1 \leq j \leq n$, let us define 
\begin{equation*}
  m_1(j) := \max \left\{0 \leq m \leq n-j: \bar{\theta^*}_{j, j+m} - \theta^*_j \leq
  \frac{2 \sigma}{\sqrt{m+1}}\right\}
\end{equation*}
and
\begin{equation*}
  m_2(j) := \max \left\{0 \leq m \leq j - 1: \theta^*_j - \bar{\theta^*}_{j-m, j}
    \leq \frac{2\sigma}{\sqrt{m+1}} \right\}. 
\end{equation*}
Note that these choices of $m_1(j)$ and $m_2(j)$ are different from
those made in the proof of Theorem~\ref{thm:L_p}. From \eqref{ann1}
and \eqref{ann2},  we obtain  
\begin{equation*}
  \E_{\theta^*} \left|\hat{\theta}_j - \theta^*_j \right|^2 \leq 16 \sigma^2 \left(
    \frac{1}{m_1(j) + 1} + \frac{1}{m_2(j) + 1} \right)
\end{equation*}
so that the risk of $\hat{\theta}$ is bounded by 
\begin{equation}\label{dag}
 R(\hat{\theta}, \theta^*) \leq \frac{16\sigma^2}{n}
  \left(\sum_{j=1}^n \frac{1}{m_1(j)+1} + \sum_{j=1}^n \frac{1}{m_2(j)
    + 1} \right) . 
\end{equation}
We shall bound $\sum_j (m_1(j) + 1)^{-1}$ below. The bound for the
term involving $m_2(j)$ will be similar. 

For each $m \geq 0$, let 
\begin{equation*}
  \rho(m) := \sum_{j=1}^n I\left\{m_1(j) = m \right\} ~~~ \text{   and
  } ~~~ l(m) := \sum_{j=1}^n I \left\{m_1(j) < m \right\}. 
\end{equation*}
Clearly $\rho(m) = l(m+1) - l(m)$ and hence
{\small \begin{equation}\label{lmm}
  \sum_{j=1}^n \frac{1}{m_1(j) + 1} = \sum_{m \geq 0}
  \frac{\rho(m)}{m+1} = \sum_{m \geq 0} \frac{l(m+1) - l(m)}{m+1} = \sum_{m \geq 0}
  \frac{l(m+1)}{(m+1)(m+2)}
\end{equation}}
because $l(0) = 0$. For $0 \leq m \leq n-1$, we can write
\begin{equation*}
  l(m+1) = \sum_{j=1}^n I\{m_1(j) < m+1 \}  \leq \sum_{j=1}^{n-m-1} I\{m_1(j) <
  m+1 \} + m + 1. 
\end{equation*}
When $1 \leq j \leq n-m-1$ and $m_1(j) < m+1$, the definition of
$m_1(j)$  implies that 
\begin{equation*}
  \bar{\theta^*}_{j, j+m+1} - \theta^*_j > \frac{2 \sigma}{\sqrt{m+2}} 
\end{equation*}
so that 
\begin{equation*}
  I\left\{m_1(j) < m+1 \right\} < \sqrt{m+2} ~\frac{\bar{\theta^*}_{j, j+m+1} -
    \theta^*_j}{2 \sigma} \qt{for $1 \leq j \leq n - m- 1$}. 
\end{equation*}
and
\begin{equation*}
  l(m+1) \leq m+1 + \frac{\sqrt{m+2}}{2\sigma} \sum_{j=1}^{n-m-1}
  \left(\bar{\theta^*}_{j, j+m+1} - \theta^*_j \right) \qt{for $0 \leq m \leq
    n-1$}. 
\end{equation*}
We now use Lemma \ref{cuz} to bound the right hand side above in terms of
$V := V(\theta^*)$. Indeed, Lemma~\ref{cuz} (applied with $l = m+2$) gives  
\begin{equation*}
  l(m+1) \leq m+1 + \frac{(m+1) \sqrt{m+2}}{2\sigma}V. 
\end{equation*}
Also, the trivial upper bound $l(m+1) \leq n$ always holds. We thus
obtain 
\begin{equation*}
  l(m+1) \leq \min \left(n, m+1 + \frac{(m+1) \sqrt{m+2}}{2\sigma}V
  \right) \qt{for all $m \geq 0$}. 
\end{equation*}
This inequality, together with \eqref{lmm}, gives 
\begin{equation*}
\sum_{j=1}^n \frac{1}{m_1(j) + 1} \leq \sum_{m \geq 0}
  \frac{1}{(m+1)(m+2)} \min \left\{n, m+1+\frac{(m+1) \sqrt{m+2}}{2
      \sigma} V\right\}.  
\end{equation*}
Now let 
\begin{equation*}
  s := \min \left(\left(\frac{\sigma n}{V}\right)^{2/3}, n \right)
\end{equation*}
and write ($C$ below stands for a positive constant whose value may
change from appearance to appearance; it does not depend on $n$, $V$
or $\sigma$)
\begin{align*}
\sum_{j=1}^n \frac{1}{m_1(j) + 1} &\leq \sum_{m \geq 0}
  \frac{1}{(m+1)(m+2)} \min \left\{n, m+1+\frac{(m+1) \sqrt{m+2}}{2
      \sigma} V\right\} \\
&\leq \sum_{0 \leq m \leq s} \frac{(m+1) + (m+1) \sqrt{m+2}(V/(2
  \sigma))}{(m+1)(m+2)} + n\sum_{m > s} \frac{1}{(m+1)(m+2)} \\
&\leq \sum_{m \leq n} \frac{1}{m+2} + \sum_{m \leq (\sigma n/V)^{2/3}}
  \frac{V}{2\sigma \sqrt{m+2}} + n \sum_{m > s}
  \frac{1}{(m+1)(m+2)} \\
&\leq C \log(en) + C \frac{V}{\sigma} \left(\frac{\sigma n}{V}
  \right)^{1/3} + \frac{C n}{s} \\
&= C \log(en) + C \frac{V^{2/3} n^{1/3}}{\sigma^{2/3}} + C n
  \max\left(\left(\frac{V}{\sigma n} \right)^{2/3} , \frac{1}{n}
  \right) \\
&\leq C \log(en) + C \frac{V^{2/3} n^{1/3}}{\sigma^{2/3}}. 
\end{align*}
One can analogously argue that 
\begin{equation*}
  \sum_{j=1}^n \frac{1}{m_2(j) + 1} \leq C \log(en) + C \frac{V^{2/3} n^{1/3}}{\sigma^{2/3}}. 
\end{equation*}
The above pair of inequalities, when combined with \eqref{dag},
complete the proof of inequality \eqref{eq:IsoWC}.

The following auxiliary lemma was used in the proof of inequality
\eqref{eq:IsoWC}. 
\begin{lemma}\label{cuz}
  For every nondecreasing sequence $a_1, \dots, a_n$ and $l \geq 1$,
  we have
  \begin{equation}\label{cuz.eq}
    \sum_{i=1}^{n-l+1} \left(\bar{a}_{i, i+l-1} - a_i \right) \leq
    \frac{l-1}{2} \left(a_n - a_1 \right) 
  \end{equation}
  where $\bar{a}_{i,j}$ for $i \leq j$ is the mean of $a_i, \dots,
  a_j$. 
\end{lemma}
\begin{proof}
  Write $b_i = a_{i} - a_{i-1}$ for $i \geq 2$ and $b_1 = a_1$. It is
  elementary to check that 
  \begin{equation*}
    \bar{a}_{i, i+l-1} = \frac{1}{l} \sum_{j=1}^{l-1} (l-j) b_{i+j}
    \qt{for $i = 1, \dots, n-l+1$}. 
  \end{equation*}
  As a result, the right hand side of~\eqref{cuz.eq} can be easily
  seen to be a linear combination of $b_2, \dots, b_n$ which are all
  nonnegative because $a_1, \dots, a_n$ is nondecreasing. The maximum
  value of the coefficient of any $b_j$ is $(l-1)/2$ which completes
  the proof. 
\end{proof}

\subsection{Proof of~\eqref{eq:IsoAdapBd2}}\label{Bellec-Adap} 
It is now assumed that $\eps_1, \dots,
\eps_n$ are i.i.d.~$N(0, \sigma^2)$. Let $Z_i
:= \eps_i/\sigma, i = 1, \dots, n$ so that $Z_1, \dots, Z_n$ are i.i.d.~$N(0, 1)$. We will use the same ideas as in the proof of Theorem~\ref{thm:AdapIsoReg}. However, when $Z \sim N_n(0, I_n)$,~\eqref{kmm} holds with equality and without the multiplicative factor 4, for every $n \geq 1$. This can be proved via symmetry
arguments based on the theory of finite reflection groups. A sketch of this argument can be found in \cite[Subsection D.4]{LivEdge}. This
observation therefore completes the proof of \eqref{eq:IsoAdapBd2}. 

\subsection{Proof Sketch of~\eqref{rlb}}  \label{prlb}
To see why \eqref{rlb} is true, first consider the case $k = 1$ where we can take $\theta^*$ to be the zero vector without loss of generality. We then have (by the formula \eqref{eq:IsoCharac}) 
\begin{equation*}
  R^{(p)} (\hat{\theta}, \theta^*) \geq \frac{1}{n} \E \left|\hat{\theta}_n \right|^p = \frac{1}{n} \E \left|\max_{u \leq n} \frac{\sum_{i = u}^n y_i}{n - u + 1} \right|^p \geq \frac{1}{n} \E \left(y_n \right)_+^p = \sigma^p \frac{C_p}{n}  
\end{equation*}
where we have used the fact that when $\theta^* = 0$, the random variable $y_n$ is normal with mean  zero and variance $\sigma^2$. This proves \eqref{rlb} for $k = 1$. For general $k$, fix $\theta^* \in \I_k$ and let $n_1, \dots, n_k$ denote the lengths of the $k$ constant blocks of $\theta^*$. Let $s_i := n_1 + \dots n_i$ for $1 \leq i \leq k$. Now if the $k-1$ jumps of $\theta^*$ are all very large in magnitude, then with high probability, 
\begin{equation*}
  \hat{\theta}_{s_i} = \max_{u \leq s_i} \min_{v \geq s_i} \frac{\sum_{l=u}^v y_l}{v - u + 1} \geq \max_{u \leq s_i} \frac{\sum_{l=u}^{s_i} y_l}{s_i - u + 1} \geq y_{s_i}. 
\end{equation*}
This can be rigorized to prove that  
\begin{equation*}
  R^{(p)}(\hat{\theta}, \theta^*) \geq \frac{1}{n}\sum_{i=1}^k \left(y_{s_i} - \theta^*_{s_i} \right)_+^p \geq C_p \sigma^p  \left(\frac{k}{n} \right) 
\end{equation*}
which yields \eqref{rlb}.

\subsection{Proof of Theorem~\ref{thm:L_p}}\label{sec:R_p-Risk}
We only need to prove the first inequality in \eqref{pab.eq}. The
second inequality is a consequence of Jensen's inequality (note that $x \mapsto
x^{(2-p)_+/2}$ is concave on $(0, \infty)$). 

We use the representation \eqref{eq:IsoCharac} of
the isotonic LSE. Fix $1 \leq j \leq n$ and let $0 \leq m \leq
n-j$. By \eqref{eq:IsoCharac}, we have 
\begin{equation*}
  \hat{\theta}_j = \min_{v \geq j} \max_{u \leq j} \bar{Y}_{uv}  \leq
  \max_{u \leq j} \bar{Y}_{u, j+m} = \max_{u \leq j}
  \left(\bar{\theta^*}_{u, j+m} + \bar{\eps}_{u, j+m} \right). 
\end{equation*}
Because $\theta^* \in \I$, we have $\bar{\theta^*}_{u, j+m} \leq
\bar{\theta^*}_{j, j+m}$ for all $u \leq j$. We therefore have
\begin{equation*}
  \hat{\theta}_j - \theta^*_j \leq \left(\bar{\theta^*}_{j, j+m} -
    \theta^*_j \right) + \max_{u \leq j} 
  \bar{\eps}_{u, j+m}. 
\end{equation*}
Taking positive parts on both sides and then raising to power $p$ and
taking expectations on both sides, we derive 
\begin{equation}\label{sho}
  \E_{\theta^*} \left(\hat{\theta}_j - \theta^*_j \right)_+^p \leq
  \E_{\theta^*} \left(\left(\bar{\theta^*}_{j, j+m} -
    \theta^*_j \right) + \max_{u \leq j} 
  \bar{\eps}_{u, j+m} \right)_+^p. 
\end{equation}
Let us now introduce some notation. Let $k = k(\theta^*)$ and let
$n_1, \dots, n_k$ denote the 
lengths of the $k$ constant blocks of $\theta^*$. Let $s_0 := 0$ and
let $s_i := n_1 + \dots + n_i$ for $1 \leq i \leq k$. For each $j = 1,
\dots, n$, let us define two integers $m_1(j)$ and $m_2(j)$ in the
following way: $m_1(j) = s_i - j$ and $m_2(j) = j - 1 - s_{i-1}$ when
$s_{i-1} + 1 \leq j \leq s_i$. The key is to realize that
$\bar{\theta^*}_{j, j+m_1(j)} = \theta^*_j$ for every $j$. As a
result, inequality \eqref{sho} with $m = m_1(j)$ gives
\begin{equation}\label{bsp}
  \E_{\theta^*} \left(\hat{\theta}_j - \theta^*_j \right)_+^p \leq
  \E \left(\max_{u \leq j} 
  \bar{\eps}_{u, j+m_1(j)} \right)_+^p = \sigma^p \E \left(\max_{u \leq j} 
  \bar{Z}_{u, j+m_1(j)} \right)_+^p 
\end{equation}
where, as before, $Z_j := \eps_j/\sigma$ for $j = 1, \dots,
n$. The fact that $\bar{Z}_{1,j + m_1(j)}, \dots, \bar{Z}_{j, j +
  m_1(j)}$ is a martingale allows us to deduce, via Doob's $L^p$
maximal inequality for nonnegative submartingales, that 
\begin{eqnarray}
 \qquad \E \left(\max_{u \leq j}  \bar{Z}_{u, j+m_1(j)} \right)_+^p & \leq & 
  \left(\frac{p}{p-1} \right)^p \E \left(\bar{Z}_{j, j+m_1(j)}
  \right)_+^p  \label{std} \\
  & = &\left(\frac{p}{p-1} \right)^p \E \left(\eta
  \right)^p_+ \left(\frac{1}{m_1(j) + 1} \right)^{p/2} \nonumber
\end{eqnarray}
where $\eta \sim N(0, 1)$ is a standard normal random variable. This
inequality requires $p > 1$ so we assume now that $p > 1$. The
argument for controlling the left hand side above for $p = 1$ will be
given subsequently.  

We have therefore proved that, for $p > 1$, 
\begin{equation*}
   \E_{\theta^*} \left(\hat{\theta}_j - \theta^*_j \right)_+^p \leq
   \sigma^p \left(\frac{p}{p-1} \right)^p \E \left(\eta
  \right)^p_+ \left(\frac{1}{m_1(j) + 1} \right)^{p/2}
\end{equation*}
for every $1 \leq j \leq n$. A similar argument gives
\begin{equation*}
   \E_{\theta^*} \left(\hat{\theta}_j - \theta^*_j \right)_-^p \leq
   \sigma^p \left(\frac{p}{p-1} \right)^p \E \left(\eta
  \right)^p_- \left(\frac{1}{m_2(j) + 1} \right)^{p/2}. 
\end{equation*}
Putting the above two inequalities together, we obtain
\begin{equation*}
   \E_{\theta^*} \left|\hat{\theta}_j - \theta^*_j \right|^p \leq
  \frac{ \sigma^p}{2} \left(\frac{p}{p-1} \right)^p \E |\eta|^p  \left\{\left( \frac{1}{m_1(j) + 1} \right)^{p/2}
    +\left(\frac{1}{m_2(j) + 1} \right)^{p/2}  \right\}
\end{equation*}
for every $j = 1, \dots, n$. Note now that 
\begin{equation*}
  \sum_{j=1}^n \left(\frac{1}{m_1(j) + 1} \right)^{p/2}  =
  \sum_{i=1}^k \sum_{j=s_{i-1} + 1}^{s_i} \left(\frac{1}{s_i - j+1}
  \right)^{p/2} = \sum_{i=1}^k \sum_{j=1}^{n_i} \left(\frac{1}{j}
  \right)^{p/2} 
\end{equation*}
and the same bound holds for $\sum_j (m_2(j) + 1)^{-p/2}$ as
well. This gives 
\begin{equation}\label{sha}
 R^{(p)}(\hat{\theta}, \theta^*) = \frac{1}{n} \E_{\theta^*} \sum_{i=1}^n \left|\hat{\theta}_j -
    \theta^*_j \right|^p \leq  \frac{\sigma^p}{n} \left(\frac{p}{p-1} \right)^p
  \E |\eta|^p \sum_{i=1}^k \sum_{j=1}^{n_i} \left(\frac{1}{j}
  \right)^{p/2}. 
\end{equation}
We now consider the two cases $1 < p < 2$ and $2 < p < \infty$
separately. The simple bound 
\begin{equation*}
  \sum_{j=1}^n j^{-p/2} \leq \frac{2}{2-p} n^{1-(p/2)} \quad \mbox{for $1 \leq
    p < 2$} 
\end{equation*}
gives 
\begin{equation*}
   R^{(p)}(\hat{\theta}, \theta^*) \leq \frac{2}{2-p} \frac{\sigma^p}{n} \left(\frac{p}{p-1}
   \right)^p \E |\eta|^p \sum_{i=1}^k n_i^{1 - (p/2)}
\end{equation*}
which proves \eqref{pab.eq} for $1 < p < 2$. 

For $p > 2$, we simply use $\sum_{i=1}^k \sum_{j=1}^{n_i} \left(\frac{1}{j} \right)^{p/2} \leq k
  \sum_{j=1}^{\infty} j^{-p/2}$
in \eqref{sha}. This completes the proof of \eqref{pab.eq} for the case
when $p > 1, p \neq 2$. 

For $p = 1$, we shall bound the expectation in the right hand side of
\eqref{bsp} in the following way.  Let $\tau^2 := 1/(m_1(j) + 1)$ be
the variance of $\bar{Z}_{j, j+m_1(j)}$. 
\begin{equation*}
  \E \left( \max_{u \leq j} \bar{Z}_{u, j+m_1(j)} \right)_+ \leq \tau
  + \int_{\tau}^{\infty} \p \left\{ \max_{u \leq j} \left(\bar{Z}_{u,
        j+m_1(j)} \right)_+ 
    \geq t \right\} dt
\end{equation*}
Because $\bar{Z}_{u, j+m_1(j)}, u = 1, \dots, j$ is a martingale, the
sequence $(\bar{Z}_{u, j + m_1(j)})_+$ is a nonnegative submartingale
and hence Doob's maximal inequality gives
\begin{equation*}
  \p \left\{ \max_{u \leq j} \left(\bar{Z}_{u,
        j+m_1(j)} \right)_+   \geq t \right\}  \leq \frac{1}{t} \E \left[ \left(\bar{Z}_{u,
        j+m_1(j)} \right)_+   \left\{ \left(\bar{Z}_{u,
        j+m_1(j)} \right)_+   \geq t \right\} \right]
\end{equation*}
so that 
\begin{align*}
  \E \left( \max_{u \leq j} \bar{Z}_{u, j+m_1(j)} \right)_+ &\leq \tau
  + \E \left(\bar{Z}_{u,
        j+m_1(j)} \right)_+ \int_{\tau}^{\infty} \left\{ \left(\bar{Z}_{u,
        j+m_1(j)} \right)_+   \geq t \right\} \frac{dt}{t} \\
&= \tau
  +  \tau \E \left[ \eta_+  \int_{1}^{\infty} \left\{ \eta_+  \geq t
  \right\} \frac{dt}{t}  \right] \leq \tau + \tau \E \eta_+^2
  \int_1^{\infty} \frac{dt}{t^2} \\
  & \leq \frac{3\tau}{2} =
  \frac{3}{2(m_1(j) + 1)}
\end{align*}
where, as before, $\eta$ is a standard normal random variable. Using
the above inequality in place of \eqref{std} proves \eqref{pab.eq} for $p
= 1$ thereby completing the proof of Theorem \ref{thm:L_p}. \qed

\subsection{Proof of Theorem \ref{pro}} \label{ppro} 
The proof of Theorem \ref{pro} basically follows from the same ideas as in the proof of Theorem \ref{thm:L_p}. Here we only highlight the changes that need to be made to the proof of Theorem \ref{thm:L_p}.  Start with inequality \eqref{sho} and using $(a + b)^p \leq C_p (a^p + b^p)$ (for example, $C_p$ can be taken to be $2^p$), obtain
\begin{equation*}
  \E_{\theta^*} \left(\hat{\theta}_j - \theta^*_j \right)_+^p \leq
C_p \left(\bar{\theta^*}_{j, j+m} -
    \theta^*_j \right)^p + C_p \E_{\theta^*} \left( \max_{u \leq j} 
  \bar{\eps}_{u, j+m} \right)_+^p.   
\end{equation*}
Now we fix an interval partition $\pi = (n_1, \dots, n_k)$ with $k(\pi) = k$ and take $s_0 := 0$ and $s_i := n_1 + \dots + n_i$ for $1 \leq i \leq k$. As before we take $m_1(j) = s_i - j$ and $m_2(j) = j - 1 - s_{i-1}$ whenever $s_{i-1} + 1 \leq j \leq s_i$. We apply the above inequality with $m = m_1(j)$. Note then that 
\begin{equation*}
  \bar{\theta^*}_{j, j+m_1(j)} - \theta^*_j  \leq V_{\pi}(\theta^*). 
\end{equation*}
The second term involving the errors is dealt with in the same way as in the proof of Theorem \ref{thm:L_p}. All other details follow just as in the proof of Theorem \ref{thm:L_p}.

\subsection{Proof of Lemma~\ref{lem:AM-Oracle}}\label{sec:AddMdl}
We shall prove Lemma \ref{lem:AM-Oracle} for the case $d = 2$ for simplicity of notation. The generalization to arbitrary $d$ is straightforward. 

Let $n_j$ be the cardinality of $\X_j$ for $j = 1, 2$ and note then that $n = n_1 n_2$. Assumption~\eqref{iden} becomes 
\begin{equation*}
  \frac{1}{n_1} \sum_{i_1 \in \X_1} f_1^*(i_1) = \frac{1}{n_2}
  \sum_{i_2 \in \X_2} f_2^*(i_2) = 0. 
\end{equation*}
The key observation is that for every pair of functions $f_1$ and $f_2$
satisfying 
\begin{equation}\label{kid}
  \frac{1}{n_1} \sum_{i_1 \in \X_1} f_1(i_1) = \frac{1}{n_2}
  \sum_{i_2 \in \X_2} f_2(i_2) = 0,
\end{equation}
and $\mu \in \R$, the following quantity
\begin{equation*}
  \sum_{i_1 \in \X_1} \sum_{i_2 \in \X_2} \big(y_{i_1, i_2} - \mu -
    f_1(i_1) - f_2(i_2) \big)^2 
\end{equation*}
is equal to 
\begin{align*}
  \sum_{i_1 \in \X_1} \sum_{i_2 \in \X_2} & \left(y_{i_1, i_2} -
  \bar{y}_{i_1, \cdot} - 
    \bar{y}_{\cdot, i_2} + \bar{y}_{\cdot, \cdot} \right)^2 +
 n_2 \sum_{i_1 \in \X_1} \left(\bar{y}_{i_1, \cdot} - \bar{y}_{\cdot, \cdot} -
   f_1(i_1) \right)^2 \\
&+ n_1 \sum_{i_2 \in \X_2} \left(\bar{y}_{\cdot, i_2} - \bar{y}_{\cdot, \cdot} -
   f_2(i_2) \right)^2 + n_1 n_2 \left(\bar{y}_{\cdot, \cdot} - \mu
  \right)^2 
\end{align*}
where $\bar{y}_{i_1, \cdot} := \left(\sum_{i_2 \in \X_2} y_{i_1, i_2} \right)/n_2$ and $\bar{y}_{\cdot, i_2} := \left(\sum_{i_1 \in \X_1} y_{i_1, i_2} \right)/n_1$. Also the overall mean is $\bar{y}_{\cdot, \cdot} := \left( \sum_{i_1 \in \X_1} \sum_{i_2 \in \X_2}
  y_{i_1, i_2} \right)/(n_1 n_2)$. This nice decomposition of the least squares criterion implies that  
\begin{equation*}
  \hat{f}_1 = \argmin \left\{\sum_{i_1 \in \X_1} \left(\bar{y}_{i_1, \cdot} - \bar{y}_{\cdot, \cdot} -
   f_1(i_1) \right)^2 : f_1 \in \F_1, \sum_{i_1 \in \X_1}
 f_1(i_1) = 0 \right\} 
\end{equation*}
and
\begin{equation*}
  \hat{f}_2 = \argmin \left\{\sum_{i_2 \in \X_2} \left(\bar{y}_{\cdot, i_2} - \bar{y}_{\cdot, \cdot} -
   f_2(i_2) \right)^2 : f_2 \in \F_2, \sum_{i_2 \in \X_2}
 f_2(i_2) = 0 \right\} 
\end{equation*}
and also that $\hat{\mu} = \bar{y}_{\cdot, \cdot}$. 

The above decomposition of the least squares criterion holds for every
$\mu, f_1$ and $f_2$ satisfying \eqref{kid}. In particular, it holds
when $f_1$ (resp. $f_2$) is replaced by $f_1^*$ (resp. $f_2^*$). It
follows therefore that $\hat{f}_1 = \hat{f}_1^{OR}$ and $\hat{f}_2 =
\hat{f}_2^{OR}$.

\subsection{Proof of Theorem~\ref{thm:PO}}
The proof of the above result uses the following lemma, which we prove in Section~\ref{sec:Proof-PO}.
\begin{lemma}\label{eq:TFunc}
Let $f$ be a continuous real-valued function defined on $\mathbb{R}^d$ ($d \ge 1$). Define the transformation $T_{k,\epsilon}: \mathbb{R}^n \rightarrow \mathbb{R}^n$ for $k \in \{1,2,\ldots, d\}$ and $\epsilon \in \mathbb{R}$ as $$ T_{k,\epsilon}(z_1,\ldots, z_d) := (z_1,\ldots, z_{k-1}, z_k + \epsilon, z_{k+1} - \epsilon, z_{k+2}, \ldots, z_d).$$ Then, $f$ is order preserving with respect to the partial order $\precsim$ if and only if for every $z \in \R^d$, and $k = 1,\ldots, d$, $f(T_{k,\epsilon}(z))$ is nondecreasing in $\epsilon$.
\end{lemma}
Fix $z \in \R^d$. Since $f$ is differentiable and order preserving, by Lemma~\ref{eq:TFunc} this is equivalent to 
\begin{eqnarray*} 
\frac{d}{d \epsilon} f(T_{k,\epsilon}(z)) & \ge &  0, \\
\mbox{i.e.,} \;\;\; f_{(k)}(T_{k,\epsilon}(z)) - f_{(k+1)}(T_{k,\epsilon}(z)) & \ge & 0,
\end{eqnarray*}
for all $k \in \{1,2, \ldots, d-1\}$. Letting $\epsilon$ go to zero yields the desired result. \qed

\subsection{Proof of Lemma~\ref{eq:TFunc}}\label{sec:Proof-PO}
Suppose that $f$ is order preserving, i.e., if $u \precsim v$, where $u,v \in \R^d$, then $f(u) \le f(v)$. Let $z \in \R^d$ and $\epsilon < \epsilon' \in \R$. We want to show that $f(T_{k,\epsilon}(z)) \le f(T_{k,\epsilon'}(z))$, for $k = 1,\ldots, d$. Fix $k \in \{1,2,\ldots, d\}$ and define $$u := T_{k,\epsilon}(z) \;\;\; \mbox{ and } \;\;\; v := T_{k,\epsilon'}(z).$$ Therefore, $\sum_{i=1}^j u_i = \sum_{i=1}^j v_i$, for $j = 1,\ldots, k-1$, $\sum_{i=1}^k u_i < \sum_{i=1}^k v_i$, and $\sum_{i=1}^j u_i = \sum_{i=1}^j v_i$, for $j = k+1,\ldots, d$. Thus, $u \precsim v$ and we have $f(u) \le f(v)$, i.e., $f(T_{k,\epsilon}(z)) \le f(T_{k,\epsilon'}(z))$.

Now suppose that $f(T_{k,\epsilon} (z))$ is nondecreasing in $\epsilon$, for every $z \in \R^d$. Take $u \equiv (u_1,\ldots, u_d) \precsim (v_1,\ldots, v_d) \equiv v$ in $\R^d$. Let $\epsilon_k := \sum_{i=1}^k (v_i - u_i)$, for $k = 1,\ldots,d$. By definition, $\epsilon_k \ge 0$, for all $k$. Define 
$$ s_k := T_{k,\epsilon_k}(s_{k-1}),$$
for $k = 1,\ldots, d$, where $s_0 \equiv u$. Note that $s_k \in \R^d$ for every $k = 1,\ldots, d$ and $s_d = T_{d,\epsilon_{d}}(s_{d-1}) = v$. Then from the nondecreasing property of $f(T_{d,\epsilon}(\cdot))$ (and using the fact that $\epsilon_k \ge 0$ for all $k$), 
\begin{eqnarray*}
f(u) & \equiv & f(T_{1,0}(s_0)) \le f(T_{1,\epsilon_1}(s_0)) = f(T_{2,0}(s_1)) \le f(T_{2,\epsilon_2}(s_1)) = f(T_{3,0}(s_3)) \\
& &  \le \cdots \le f(T_{d,0}(s_{d-1})) \le f(T_{d,\epsilon_{d}}(s_{d-1})) \equiv f(v),
\end{eqnarray*} yielding the desired result. 
\qed

\subsection{Outline of a proof of~\eqref{eq:LimDist}}\label{sec:Asym-IsoReg} 
In the following we sketch an outline of a proof of~\eqref{eq:LimDist}. Our proof technique directly appeals to the characterization of the isotonic LSE as described at the beginning of this section; see~\cite[Section 3.2.15]{VW96} for an alternative proof technique that uses the {\it switching relationship}, due to Groeneboom~\cite{G85}. Let us further assume that the i.i.d.~errors $\varepsilon_i$'s have a finite moment generating function near 0. This assumption lets us avoid the use of heavy empirical process machinery and, we hope, will make the main technical arguments simple and accessible to a broader audience.

We consider the stochastic process $$\Z_n(h) := n^{2/3} [F_n(t + n^{-1/3}h) - F_n(t) - n^{-1/3} h f(t)], $$ for $h \in [-tn^{1/3} , (1-t)n^{1/3}]$. We regard stochastic processes as random elements in $D(\R)$, the space of right continuous functions on $\R$ with left limits, equipped with the projection $\sigma$-field and the topology of uniform convergence on compacta; see~\cite[Chapters IV and V]{Pollard84}   for background.

Observe that if $u$ is a bounded function and $v$ is affine then $\widetilde{u + v} = \tilde{u} + v$. Using this, $\tilde \Z_n$, the largest convex function sitting below $\Z_n$, has the form $$\tilde \Z_n(h) = n^{2/3} [\tilde F_n(t + n^{-1/3}h) - F_n(t) - n^{-1/3} h f(t)], $$ for $h \in [-tn^{1/3} , (1-t)n^{1/3}]$. By taking the left-derivative of the above process at $h=0$ we get (w.p.~1),
\begin{equation}\label{eq:Delta_n}
\Delta_n =  [\tilde \Z_n]'(0).
\end{equation}
The above relation is crucial, as it relates $\Delta_n$, the quantity of interest, to a functional of the process $\Z_n$. We study the process $\Z_n$ (and show its convergence) and apply a (version of) `continuous' mapping theorem (see e.g.,~\cite[pp.~217-218]{KP90}) to derive the limiting distribution of $\Delta_n$.

Let $\check F_n:[0,1] \to \R$ be the continuous piecewise affine function (with possible knots only at $i/n$, for $i = 1,\ldots, n$) with
\begin{equation*}
\check F_n \Big(\frac{i}{n} \Big) := \frac{1}{n} \sum_{j=1}^i f\Big(\frac{j}{n} \Big), \quad \mbox{for} \; i = 0,\ldots, n,
\end{equation*}
and let $F:[0,1] \to \R$ be defined as
\begin{equation*}
F(x) := \int_0^x f(s) ds.
\end{equation*}
To study the stochastic process $\Z_n$ we decompose $\Z_n$ into the sum of the following three terms:
\begin{eqnarray*}
\Z_{n,1}(h) & := & n^{2/3} [F_n(t + n^{-1/3}h) - \check F_n(t + n^{-1/3}h) - F_n(t) + \check F_n(t)], \\
\Z_{n,2}(h) & := & n^{2/3} [\check F_n(t + n^{-1/3}h) - F(t + n^{-1/3}h) - \check F_n(t) + F(t)], \\
\Z_{n,3}(h) & := & n^{2/3} [F(t + n^{-1/3}h) - F(t) - n^{-1/3} h f(t)],
\end{eqnarray*}
Observe that $F_n - \check F_n$ is just the partial sum process, properly normalized. By the Hungarian embedding theorem (see e.g.,~\cite{KMT76}) we know that the partial sum process is approximated by a Brownian motion process such that
\begin{equation}\label{eq:StrongApp}
F_n(x) - \check F_n(x) = n^{-1/2}  \sigma \B_n(x) + R_n(x),
\end{equation} 
where $\B_n$ is a Brownian motion on [0,1] and $$\sup_x |R_n(x)| = O \Big(\frac{\log n}{n}\Big) \; \mbox{ w.p.~1}.$$ Thus, 
$$\Z_{n,1}(\cdot) = \sigma \W_n(\cdot) + o_p(1),$$ where the process $\W_n$ is defined as $\W_n(h) := n^{1/6} \{\B_n(t + n^{-1/3} h) - \B_n(t)\}, h \in \R$, and $\W_n \sim \W$ with $\W$ being distributed as a two-sided Brownian motion (starting at 0). This shows that the process $\Z_{n,1}$ converges in distribution to $\W$.

To study $\Z_{n,2}$, observe that as $f(\cdot)$ is continuously differentiable in a neighborhood $\N$ around $t$, we have (by a simple interpolation bound) $$\sup_{x \in \N} |\check F_n(x) - F(x)| = O(n^{-1}).$$ Thus, $\Z_{n,2}$ converges to the zero function. By a simple application of Taylor's theorem, we can show that $\Z_{n,3}$ converges, uniformly on compacta, to the function $D(h) := h^2 f'(t)/2$.

Combining the above results, we obtain that $\Z_n$ converges in distribution to the process $\Z(h):= \sigma \W(h) + h^2 f'(t)/2$, i.e., $$\Z_n \stackrel{d}{\to} \Z$$ in the topology of uniform convergence on compacta. Thus, it is reasonable to expect that $$\Delta_n = [\tilde \Z_n]'(0) \stackrel{d}{\to} [\tilde \Z]'(0).$$ However, a rigorous proof of the convergence in distribution of $\Delta_n$ involves a little more than an application of a continuous mapping theorem. The convergence of $\Z_n$ to $\Z$ is only under the metric of uniform convergence on compacta. However, the GCM near the origin might be determined by values of the process far away from the origin; the convergence $\Z_n$ to $\Z$ itself does not imply the convergence of $[\tilde \Z_n]'(0)$ to $[\tilde \Z]'(0)$. We need to show that $\Delta_n$ is determined by values of $\Z_n(h)$ for $h$ in an $O_p(1)$ neighborhood of 0; see e.g.,~\cite[pp.~217-218]{KP90} for such a result with a detailed proof.

It can be shown that $[\tilde \Z]'(0) \stackrel{d}{=} \kappa \mathbb{C}$ (see e.g.,~\cite[Chapter 3.2]{GJ14} and~\cite[Exercise 3.12]{GJ14}) which completes the proof sketch of~\eqref{eq:LimDist}.

\subsection{A sketch of the main steps in the proof of~\eqref{eq:CvxLim}}\label{sec:Cvx-Reg-Asymp}
Let $\hat \Theta_n: [0,1] \to \R$ be the continuous piecewise affine function  (with possible knots only at $i/n$) such that $\hat \Theta_n(i/n) = \hat \Theta_i$, for $i = 1,\ldots, n$, and $\hat \Theta_n(0)=0$. Let $\tilde \Theta_n: [0,1] \to \R$ be defined as $$\tilde \Theta_n(x) := \int_0^x \hat f_n(s) \, ds.$$ Note that $\hat \Theta_n$ and $\tilde \Theta_n$ are asymptotically the same, but it is easier to study $\tilde \Theta_n$. Let $G_n$ denote the empirical distribution function of the design points $\{i/n:i =1,\ldots, n\}$. Further, let us define the stochastic processes $\h_n^{loc}$ and $\tilde \h_n^{loc}$ (on $\R$) as
{\small \begin{eqnarray*}
\h_n^{loc}(h)  :=  n^{4/5} \int_t^{t + n^{-1/5}h} \left[ \hat \Theta_n(v) - \hat \Theta_n(t) - \int_t^v \big\{ f(t) - (u-t) f'(t) \big\} d G_n(u)\right] dv + A_n h + B_n, \\
\tilde \h_n^{loc}(h) :=  n^{4/5} \int_t^{t + n^{-1/5}h} \left[ \tilde \Theta_n(v) - \tilde \Theta_n(t) - \int_t^v \big\{ f(t) - (u-t) f'(t) \big\} du\right] dv + A_n h + B_n, \qquad 
\end{eqnarray*}}
where $$A_n := n^{3/5} \{\hat \Theta_n(t) - F_n(t)\}, \quad \mbox{and} \quad B_n := n^{4/5}\int_0^t  \{  \hat \Theta_n(s) - F_n(s) \}\; ds.$$ The process $\h_n^{loc}$ can be thought of as the `localization' of the left side of~\eqref{eq:CvxLSE-Ch}. The process $\tilde \h_n^{loc}$ is important for the following reason. As $\tilde \Theta_n'(x) = \hat f_n(x)$, for $x \in (0,1)$, by differentiating the process $\tilde \h_n^{loc}$ twice we get $$(\tilde \h_n^{loc})''(h) = n^{2/5} \{\hat f_n(t + n^{-1/5} h) - f(t) - f'(t) n^{-1/5} h\}.$$ Thus, the quantity of interest $\Delta_n$ is related to the process $\tilde \h_n^{loc}$ as $$\Delta_n = (\tilde \h_n^{loc})''(0),$$ and this motivates the study of the process $\tilde \h_n^{loc}$ (cf.~\eqref{eq:Delta_n}). Further, one can show that $\tilde \h_n^{loc}$ and $\h_n^{loc}$ are asymptotically the same; see~\cite[p.~1696]{GJW01-a}.

The following process can be thought of as the localization of the right side of~\eqref{eq:CvxLSE-Ch}:
{\small \begin{eqnarray*}
\Z_n^{loc}(h) := n^{4/5} \int_t^{t + n^{-1/5}h} \left[ F_n(v) - F_n(t) - \int_t^v \big\{ f(t) - (u-t) f'(t) \big\} d G_n(u)\right] dv.
\end{eqnarray*}}
Moreover, as we will show, the process $\Z_n^{loc}$ converges to a limiting distribution and is related to $\h_n^{loc}$ (and thus to $\tilde \h_n^{loc}$) as, for all $h \in \R$,
\begin{eqnarray*}
\h_n^{loc}(h) - \Z_n^{loc}(h) & = & n^{4/5} \int_t^{t + n^{-1/5}h} \{\hat \Theta_n(v) - F_n(v)\} dv + B_n \\
 & = & n^{4/5} \int_0^{t + n^{-1/5}h} \{\hat \Theta_n(v) - F_n(v)\} dv  \ge 0,
\end{eqnarray*}
with equality if $t + n^{-1/5} h$ is a kink point. As the process $\Z_n^{loc}$ involves empirical averages we can use empirical process techniques (or the Hungarian embedding theorem; see~\eqref{eq:StrongApp}) and Taylor's expansion to show that $$\Z_n^{loc}(h) \stackrel{d}{\to} \sqrt{f(t)} \int_0^h \W(s) ds + \frac{1}{24} f''(t) h^4 := \Z(h),\quad h \in [-K,K],$$ as a stochastic process in $D[-K,K]$, for any $K >0$, under the metric of uniform convergence on $[-K,K]$,  where $\W$ is a two-sided Brownian motion; see~\cite[pp.~1694--1696]{GJW01-a}. 

By~\cite[Theorem 6.1]{GJW01-a} there exists an almost surely uniquely defined random continuous process $\h$, called an invelope of the process $\Z$, such that (i) $\h(h) \ge \Z(h)$, for each $h \in \R$; (ii)  $\h$ has a convex second derivative, and, with probability 1, $\h$ is three times differentiable at $h=0$; and (iii) $\int \{\h(h) - \Z(h)\} d \h^{(3)}(h) = 0$  (which signifies that $\h = \Z$ on the set where $\h^{(3)}$ has a jump). Further, it is shown in~\cite[pp.~1689--1692]{GJW01-a} that along with the process $\Z_n^{loc}$, the ``invelope'' $\h_n^{loc}$ converges in such a way that the second and third derivatives of $\h_n^{loc}$ at zero converges in distribution to the corresponding quantities of $\h$; also see~\cite{GJW01-b}. Thus, $\tilde \h_n^{loc}$ converges weakly to $\h$ (as $\tilde \h_n^{loc}$ and $\h_n^{loc}$ are asymptotically equivalent) and $\Delta_n = (\tilde \h_n^{loc})''(0)$ (see~\eqref{eq:CvxLim}) converges weakly to $\h''(0)$.

\bibliographystyle{abbrv}
\bibliography{AG}
\end{document}